\documentclass[letterpaper,journal]{IEEEtran}
\usepackage{amsmath,amsfonts}
\usepackage{algorithmic}
\usepackage{algorithm}
\usepackage{array}
\usepackage[caption=false,font=normalsize,labelfont=sf,textfont=sf]{subfig}
\usepackage{textcomp}
\usepackage{stfloats}
\usepackage{url}
\usepackage{verbatim}
\usepackage{graphicx}
\usepackage{cite}
\usepackage{threeparttable}
\newtheorem{theorem}{Theorem}
\newtheorem{remark}{Remark}

\newtheorem{lemma}{Lemma}
\newtheorem{definition}{Definition}

\newtheorem{conjecture}{Conjecture}

\newcommand{\R}{\mathbb{R}}
\newcommand{\C}{\mathbb{C}}

\newcommand{\ba}{\mathbf{a}}
\newcommand{\bb}{\mathbf{b}}
\newcommand{\bc}{\mathbf{c}}

\newcommand{\be}{\mathbf{e}}
\newcommand{\bg}{\mathbf{g}}

\newcommand{\bu}{\mathbf{u}}
\newcommand{\bv}{\mathbf{v}}

\newcommand{\bx}{\mathbf{x}}

\newcommand{\by}{\mathbf{y}}
\newcommand{\bh}{\mathbf{h}}

\newcommand{\bA}{\mathbf{A}}

\newcommand{\bD}{\mathbf{D}}

\newcommand{\bG}{\mathbf{G}}
\newcommand{\bH}{\mathbf{H}}

\newcommand{\bS}{\mathbf{S}}

\newcommand{\bT}{\mathbf{T}}

\newcommand{\bV}{\mathbf{V}}

\newcommand{\balpha}{\boldsymbol{\alpha}}

\newcommand{\bPsi}{\boldsymbol{\Psi}}

\newcommand{\bPhi}{\boldsymbol{\Phi}}

\newcommand{\bLambda}{\boldsymbol{\Lambda}}

\DeclareMathOperator*{\rank}{rank}
\DeclareMathOperator*{\var}{Var}
\DeclareMathOperator*{\diag}{diag}

\DeclareMathOperator*{\tr}{tr}
\newcommand{\bff}{\mathbf{f}}

\newcommand{\cA}{\mathcal{A}}
\newcommand{\cM}{\mathcal{M}}
\newcommand{\cR}{\mathcal{R}}
\newcommand{\cN}{\mathcal{N}}

\newcommand{\dd}{{\rm d}}

\begin{document}
\title{Stable Takens' Embedding Theorem \\ for Non-Uniformly-Sampled Linear Systems}

\author{
    Fisher Ng\IEEEauthorrefmark{1}
    \thanks{\IEEEauthorrefmark{1} F. Ng is with the Department of Applied Mathematics, University of Washington, Seattle, WA (fisherng@uw.edu).}
    and  
    J. Nathan Kutz\IEEEauthorrefmark{2}
    \thanks{\IEEEauthorrefmark{2} J. Kutz is with Autodesk Research, London, UK (nathan.kutz@autodesk.com).}
}



\maketitle

\begin{abstract}
Takens' time-delay embedding theorem provides conditions under which delay-coordinate maps, formed using uniformly-sampled time series of trajectories evolving on attractors of dynamical systems, can faithfully represent the dynamics of the original system. 
Nonlinear systems can be highly sensitive, and Takens' theorem does not provide guarantees about the stability of time-delay embeddings.
In the linear setting, statements about the stability of time-delay embeddings are more tractable and have been proven for delay-coordinate maps with evenly-spaced delays. 
In many experimental applications, however, time series data may be non-uniformly-sampled, especially in systems with multiple timescales or when using event-based rather than time-based sampling techniques.
In this paper, we extend the theorems for the stable linear Takens' embeddings to the setting where the delay-coordinate maps involve unevenly-spaced delays. 
We pose a conjecture about the rank of generalized Vandermonde matrices that capture the temporal structure of time-delay embeddings.
We prove that, provided the conjecture holds, existing theorems about stable linear Takens' embeddings readily extend to unevenly-sampled settings, and the quality of the embedding converges to the same asymptotic bounds when using a large number of delays in the delay-coordinate map, a result which is supported by numerical simulations.
\end{abstract}


\section{Introduction}
\IEEEPARstart{T}{he} information available about many dynamical systems of interest comes in the form of partial observations of the full state captured in sequences of time snapshots known as time series.
To model the original dynamics, delay-coordinate maps--vectors consisting of the observed states and their consecutive time-lagged observations-- are often created to represent the dynamics of the system inherent in the time series data in a latent ``delay-coordinate" space. 
Takens \cite{Takens_1981} showed that, for solution trajectories of a dynamical system evolving on a low-dimensional attractor, under mild conditions a generic delay-coordinate map is an embedding: a smooth one-to-one map the derivative of which is also one-to-one.
Embeddings preserve all relevant topological and differential information of the original system, which for dynamical systems implies distinct states remain distinguishable under the mapping, as does the evolution of their solution trajectories.
Takens' time-delay embedding theorem thus provides a justification for why it is reasonable to assume that incomplete observations of a system taken over time can be used to reasonably model the complete, original state dynamics of a system at a particular instant in time.

To be more precise, consider a dynamical system $\dot{\bx}(t) = \bff(\bx(t))$ with state $\bx(t) \in \R^n$ evolving on an attractor $\cM$ of dimension $d$ with $d \le n$, observed at intervals of $\tau$ by a single measurement apparatus modeled by a smooth observation function $h: \R^n \to \R$. 
Takens \cite{Takens_1981} and later Sauer et al. \cite{SYC_1991} showed that delay-coordinate maps $\bPsi: \R^n \to \R^\ell$ of the form
\begin{align} \label{eq:Delay_Coordinate_Map}
    \bPsi_{\{h, \tau, \ell\}}(\bx(t))
    = \begin{bmatrix}
        h(\bx(t)) \\ h(\bx(t - \tau)) \\ \vdots \\ h(\bx(t - (\ell-1) \tau)
    \end{bmatrix}
\end{align}
preserve, with probability one, the topology of the original dynamics on the attractor, provided $\ell \ge 2d+1$ and $\tau$ satisfies some conditions relating to the periodic orbits of the solutions.

\begin{figure*}[t]
    \centering
    \includegraphics[width=0.9\textwidth]{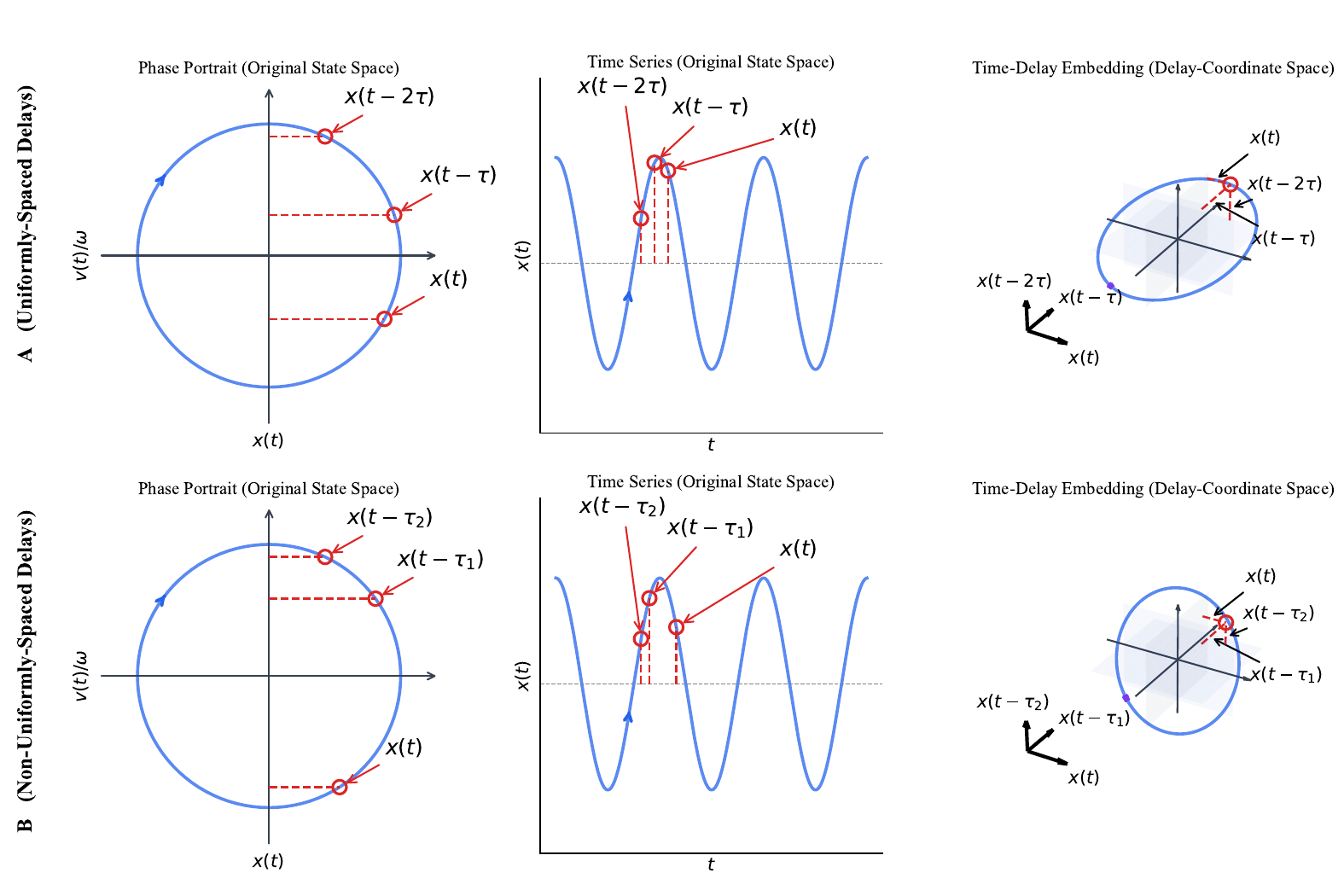}
    \caption{An undamped harmonic oscillator $\ddot{x}(t) + \omega x(t) = 0$ with solutions $x(t) = \cos(\omega t)$ and $v(t) = -\omega \sin(\omega t)$ is a class $\cA(1)$ system (see Definition \ref{def:Class_Ad_Sys}).
    For this setting, $\omega = 2\pi$.
    For $x(t)$ being the observed state variable (with the observation function being $h: [x(t), v(t)] \to x(t)$, i.e. measuring only the position and not velocity), we plot the position versus normalized velocity phase portraits, the time series for the position function, and the $x(t)$-based time-delay embeddings using $\ell = 3$ delays, depicted for (A) the case of uniformly-spaced delays in the delay-coordinate map with $\tau = 0.13$ $(x(t), x(t - \tau), x(t - 2 \tau))$, and (B) the case of the non-uniformly-spaced delays with $\tau_1 = 0.08$ and $\tau_2 = 0.34$ $(x(t), x(t - \tau_1), x(t - \tau_2))$.
    The delay-coordinate maps preserve the topological and differential structure of the attractor in the reconstruction space, ensuring the original signal is recoverable, but differ from each other based on the choice of time delays.}
    \label{fig:Time_Delay_Embedding}
\end{figure*}

Takens' theorem is a statement of the existence and prevalence of embeddings, but, as originally stated, it does not provide metrics on the stability of mappings for particular $h$ and $\tau$. 
Because Takens formulated his time-delay embedding theorem for general nonlinear systems observed with nonlinear observation functions, the nonlinearity makes it difficult, although not impossible--as shown by \cite{Yap_2014, Eftekhari_2018}--to make statements about the stability of time-delay embeddings.
For linear dynamical systems evolving on their attractors, which consists of fixed points or oscillatory motion, a stability analysis becomes feasible, as was performed by Rozell and Yap \cite{Yap_2010, Yap_2011}. 
In both the linear and nonlinear settings, the stability of a time-delay embedding depends on the geometry of the attractor, the number of delay coordinates used, and the extent to which the set of observation functions used is able to observe all the dimensions of the given attractor.

While in most applications the system is observed at evenly-spaced intervals, leading to delay-coordinate maps with evenly-spaced lags, there are a number of applications where non-uniformly-spaced delays are of interest.
Thus, an important question is whether delay-coordinate maps of linear systems constructed with unevenly-spaced lags are still embeddings and, if so, under what conditions they are stable and to what degree they are stable.

Several works have given attention to the reconstruction of signals using non-uniformly sampled signals \cite{Sommen_2008, Vlanchos_2010, Faes_2011, Han_2018, Tao_2018}.
Non-uniform delay-coordinate maps are valuable in multi-scale modeling where systems exhibit dynamics at multiple different timescales \cite{Judd_1998, Hirata_2006, Gao_2021, Tan_2023}. 
In systems with fast-slow time dynamics, using the slow timescale as the single, uniform delay may limit the delay-coordinate map's \eqref{eq:Delay_Coordinate_Map} ability to resolve the fast dynamics in the embedding, as the fast dynamics may appear as noise in the reconstruction space \cite{Tan_2023}. 
Thus, allowing delay-coordinate maps to combine delays of fast and slow frequencies may improve the embedding quality.

In other situations, while the system itself may not involve multi-scale dynamics, the measurement apparatus may sample the system at non-uniformly-spaced intervals, especially when considering mobile sensors or event-based rather than time-based sensing.
Sauer \cite{Sauer_1994}, for example, proposed an event-based ``interspike interval" detection mechanism, which, rather than measuring a system at set time intervals, initiates a measurement based on the state of the system.
Huke et al. \cite{Huke_2007} established theoretical results on the conditions needed for delay-coordinates using such a mechanism to be embeddings and extended it to time-based observation functions in general.

A central challenge with delay-coordinate maps with unevenly-spaced delays is that, unlike delay-coordinate maps with evenly-spaced delays that have only two parameters to optimize (namely, the number of delays $\ell$ and the time delay $\tau$), maps of non-uniform delays must optimize over $\ell$ distinct delay parameters $\{\tau_k\}_{k = 1}^\ell$.
The optimization problem is combinatorially complex, making such maps unwieldy and/or intractable.
However, a number of works have developed algorithms to address this and make such mappings relevant \cite{Tan_2023, Shen_2013, GomezGarcia_2014}.

It is the purpose of this paper to provide an extension to the stable linear Takens' embedding for linear systems developed by Yap and Rozell \cite{Yap_2011} to the case of delay-coordinate maps of non-uniform delays.
In Section II, we provide the relevant background for stable linear time-delay embeddings.
In Section III, we establish the main theoretical results for delay-coordinate maps with unevenly-spaced delays, beginning with a description of the connection between Fourier Vandermonde matrices and linear delay-coordinate maps with evenly-spaced delays.
To extend the stable linear Takens' embedding to the non-uniform sampling setting, it is necessary to consider a form of generalized Vandermonde matrices that involve non-integer and unevenly-spaced powers and that have more lags than points considered.
Just as typical Vandermonde matrices require a distinct set of points and a distinct set of consecutive, evenly-spaced integer powers to be full rank--which translate to relatively reasonable assumptions that the underlying frequencies in a given signal are not duplicated and are sampled at evenly-spaced time intervals, so too would it be reasonable to assume that, if a signal has distinct frequencies and is measured at distinct but not evenly-spaced time instants, then the generalized Vandermonde matrix modeling it ought to be full rank with probability one.
To verify the claim, we provide conditions under which such generalized Vandermonde matrices are full rank for trivial cases where the matrix involves $m = 1,2,3$ distinct points and $\ell \ge m$ time powers, and also for the asymptotic regime as $\ell \to \infty$ for a few example random time-sampling schemes, namely those that obey uniform, exponential, and normal distributions.
Empirical results also support the claim that the generalized Vandermonde matrices considered are full rank under relatively mild assumptions.
Provided such generalized Vandermonde matrices are full column rank, it is possible to then prove a theorem for the existence of stable Takens' time-delay embeddings of dynamics of linear dynamical systems using delay coordinate maps with unevenly-spaced delays, as well as another theorem providing explicit results about the limiting bounds to which the delay-coordinate maps converge as $\ell \to \infty$.
Numerical results on simulated dynamical systems support the newly-derived stable embedding theorems.

\section{Background}

\subsection{Preliminaries}
In mostly keeping with the conventions used by Yap and Rozell \cite{Yap_2011}, for a vector $\bu = [u_1, ..., u_2]^T \in \C^n$, we denote the element-wise complex conjugate as $\bu^*$, its regular transpose as $\bu^T$, and its Hermitian transpose as $\bu^H = (\bu^*)^T$. 
We use $i$ as the imaginary unit. 
We let $\tr(\cdot)$ denote the trace of a matrix--the sum of its diagonal entries, $\det(\cdot)$ the matrix determinant, and $\kappa(\cdot)$ the $2$-norm condition number, computed as the ratio of the maximal to the minimal singular value of the matrix.

To analyze the rank, condition number, and determinant of a matrix, one helpful tool is the following eigenvalue localization theorem by Gerschgorin, which, using only explicit information about the entries of a matrix, generates regions in which the set of all the eigenvalues of the matrix is guaranteed to be contained.
\medskip

\begin{theorem}\label{thm:Gerschgorin}
    (Gerschgorin Circle Theorem \cite{Gerschgorin_1931, Horn_Johnson_1985}) Let $\bS \in \C^{m \times m}$ be a matrix with entries $s_{pq}$ with eigenvalues $\{\lambda_p\}_{p = 1}^m$.
    For $p \in \{1,...,m\}$, let $R_p$ be the sum of the absolute values of the non-diagonal entries in the $p$th row of $\bS$: $R_p := \sum_{q = 1, p \neq q}^m |s_{pq}|$.
    Define the closed disc $D(s_{pp}, R_p) \subseteq \C$, centered at $s_{pp}$ with radius $R_p$, as the $p$th Gerschgorin disc.
    Then, every eigenvalue of $\bS$ lies within at least one of the Gerschgorin discs $D(s_{pp},R_{p})$. 
    Denote by $\lambda(\bS)$ the spectrum of $\bS$--the collection of all of the eigenvalues of $\bS$.
    Then $\lambda(\bS) = \bigcup_{p = 1}^m \lambda_p(\bS) \subseteq \bigcup_{p = 1}^m D(s_{pp},R_p)$.
\end{theorem}
\bigskip

By ensuring that zero is not in the eigenvalue inclusion set defined by the union of the Gerschgorin discs of a matrix, a condition known as strict diagonal dominance where $|s_{pp}| > R_p$ for all $p \in \{1,...,m\}$, it follows that zero cannot be an eigenvalue of $\bS$, and thus $\bS$ is full rank.
The extremal values of the Gerschgorin disc set offer upper and lower bounds on the eigenvalues of the matrix $\bS$, which can be used to establish upper bounds on $\kappa(\bS)$ and potentially $\det(\bS)$.

\subsection{Dissipative Linear Systems}
Consider the linear, autonomous dynamical system described by the differential equation
\begin{align}\label{eq:ODE_System}
    \dot{\bx}(t) = \bA \bx(t) \quad \mbox{s.t.} \quad \bx(0) = \bx_0
\end{align}
where $\bx(t) \in \R^n$ is the state at time $t$, $\bA \in \R^{n \times n}$ is the matrix vector field that describes the governing dynamics of the system, and $\bx_0$ is the initial condition.
The solution trajectories of \eqref{eq:ODE_System} take the form $\bx(t + s) = e^{\bA t} \bx(s)$, where $\bPhi(t) := e^{\bA t}$ is the \textit{flow matrix}--the matrix that when applied to a state at time $s$ evolves it along the solution trajectory by a time interval of duration $t$. (When $s = 0$, $\bx(t) = e^{\bA t} \bx_0$).

As an additional condition, assume the linear system \eqref{eq:ODE_System} is dissipative.
Here, we use the term dissipative in the sense that none of the eigenvalues of the state matrix $\bA$ have a positive real component so that exponential growth does not occur and that at least one eigenvalue has a negative real component so that solution trajectories decay over time in some dimension, but that purely imaginary pairs of eigenvalues are allowed so that constant oscillatory motion is allowed.
When dynamical systems are dissipative, they tend to scatter energy over time and, in their long-term asymptotic behavior, settle to a compact, low-dimensional subspace within the ambient high-dimensional state space known as an attractor.
Unlike transient solutions that might involve exponential growth or decay and thus may be high-dimensional and unbounded, solutions on the attractor are inherently low-dimensional and, due to the compactness of the attractor, are confined to certain behavior for all time.
Provided enough samples are taken of states on the attractor, the task of predicting the evolution of solution trajectories on the attractor may tractable, as predictions would be ``within distribution", so to speak, as opposed to predictions for transient solutions, which by their unbounded nature may end up out-of-distribution relative to the already available data.
Thus, that solutions on the attractor evolve in a low-dimensional and compact space and represent the long-term dynamic equilibrium of a system makes such solutions particularly interesting.

For linear systems, attractors can consist of fixed points or purely oscillatory motion, a notion which the following definition formalizes. 
\medskip

\begin{definition}\label{def:Class_Ad_Sys}
    (Def. II.1 \cite{Yap_2011}) 
    A linear dynamical system in $\R^n$ as in \eqref{eq:ODE_System} is of \textit{class} $\cA(d)$ for $d \le \frac{n}{2}$ if the system matrix $\bA$ is real, full-rank, and has distinct eigenvalues. 
    Moreover, $\bA$ has only $d$ strictly imaginary conjugate pairs of eigenvalues and the rest of its eigenvalues have real components strictly less than zero. 
    The strictly imaginary conjugate pairs of eigenvalues are called the $\cA$-eigenvalues, and they can be expressed as $\{\pm i \theta_j\}_{j = 1}^d$ where $\{\theta_j\}_{j = 1}^d > 0$ are $d$ distinct numbers. 
    The corresponding unit-norm $\cA$-\textbf{eigenvectors} are $\bv_1, \bv_1^*, \cdots, \bv_d, \bv_d^*$. 
    The corresponding eigenvalues of the flow matrix $\bPhi(t)$, the $\cA_{\bPhi}$-\textbf{eigenvalues}, and are $\{e^{\pm i \theta_j \tau}\}_{j = 1}^d$.
\end{definition}
\bigskip

Let $\bLambda = \diag(-i \theta_1, i \theta_1, ..., -i \theta_d, i \theta_d)$ be the diagonal matrix consisting of $\cA$-eigenvalues, and $\bV = \begin{bmatrix}
\bv_1 & \bv_1^* & \cdots & \bv_d & \bv_d^* \end{bmatrix} \in \C^{n \times 2d}$ be the collection of the $\cA$-eigenvectors with $\rank(\bV) = 2d$.
The flow matrix $\bPhi(t)$ and the vector field $\bA$ share the same eigenvalues $\bD(t) := e^{-\bLambda t}$. $\bPhi(t) \bV = \bV \bD(t)$.

The following definition provides a means to define the solutions that evolve on the attractor.
\medskip

\begin{definition}\label{def:Attractor}
    (Def. II.2 \cite{Yap_2010}) Let a linear dynamical system be of class $\cA(d)$ and let $\bx_0 = \bV \balpha_0 \in \R^n$ for some $\balpha_0 \in \C^{2d}$ be an arbitrary initial state of the system. We define the \textbf{attractor} of this linear dynamical system to be $\cM = \{\bx \in \R^n | \bx = \bV e^{\bLambda t} \balpha_0, t \in \R\}$.
\end{definition}
\smallskip

\subsection{Time-Delay Embedding of Linear Systems}
Suppose the dynamical system \eqref{eq:ODE_System} is sampled using a linear observation function, represented by the vector $\bh \in \R^n$, at regular intervals of length $\tau > 0$, generating measurements $\by(t) = \bh^T \bx(t) \in \C$ that form a collection $\{\by(t - k \tau)\}_{k = 0}^{\ell-1}$. 
The delay-coordinate map $\bPsi_{\{\bh, \bPhi, \tau\}}: \R^n \to \R^\ell$ is defined as:

\begin{align} \label{eq:Uniform_Delay_Coordinate_Map}
    \bPsi(\bx(t)) 
    = \begin{bmatrix}
        \by(t) \\ \by(t - \tau) \\ \vdots \\ \by(t - (\ell-1) \tau)
    \end{bmatrix}
    = \begin{bmatrix}
        \bh^T \bx(t) \\ \bh^T \bPhi^{-\tau} \bx(t) \\ \vdots \\ \bh^T \bPhi^{-(\ell - 1) \tau} \bx(t)
    \end{bmatrix}
\end{align}
For a delay-coordinate map to faithfully represent the original dynamics of the system, the observation function $\bh$ and the lag $\tau$ must be chosen to avoid aliasing and ensure the uniqueness of each state is preserved.
Ideally, we would like to have a way to quantify how accurately a particular delay-coordinate map represents the original dynamics, and the notion of stability becomes particularly important.
The following definition provides a notion of a stable embedding and a way to quantify whether nearby points in the original state space remain close to each other under the delay-coordinate map and vice versa.
\medskip

\begin{definition}
    (Def. II.3 \cite{Yap_2011}) Suppose we have a dynamical system in $\R^n$ that converges to an attractor $\cM$ and a linear map $\bPsi: \R^n \to \R^\ell$. 
    We say that $\bPsi$ is a stable embedding of $\cM$ with conditioning $\delta$ if for all $\bx_1, \bx_2 \in \cM$ and for some scaling constant $C$, we have
    \begin{align} \label{eq:Stable_Embedding}
        C(1 - \delta) \le \frac{\|\bPsi(\bx_1) - \bPsi(\bx_2)\|_2^2}{\|\bx_1 - \bx_2\|_2^2} \le C(1 + \delta).
    \end{align}
\end{definition}
\medskip

Yap and Rozell proved that delay-coordinate maps as in \eqref{eq:Uniform_Delay_Coordinate_Map} using linear observation functions for solutions of linear systems evolving on their attractors form stable embeddings \cite{Yap_2010, Yap_2011}.
The first theorem \cite{Yap_2011} proved establishes the existence of stable Takens' embeddings for linear systems.
\medskip

\begin{theorem} \label{thm:Stable_Takens_Embedding_1}
    (Linear Takens' Embedding; Theorem III.1 \cite{Yap_2010, Yap_2011}) Consider a class-$\cA(d)$ linear dynamical system in $\R^n$ in steady-state. 
    Let $\tau > 0$ be the sampling rate, $\bh \in \R^n$ be the observation function, and $\bPsi$ be the delay-coordinate map with $\ell$ delays as in \eqref{eq:Uniform_Delay_Coordinate_Map}. 
    Suppose that $\ell \ge 2d$, the $\cA_{\bPhi}$-eigenvalues $\{e^{\pm i \theta_j \tau}\}$ are distinct and purely imaginary, and $\bv_j^H \bh \neq 0$ for all $j \in \{1,...,d\}$. 
    Then, for all distinct pairs of points $\bx_1, \bx_2 \in \cM$, $\bPsi$ satisfies \eqref{eq:Stable_Embedding} for some constants $C$ and $\delta < 1$.
\end{theorem}
\bigskip

In order to ensure the system frequencies are distinct, a sufficient, although not necessary, condition is $\tau < \frac{\pi}{\theta_{\max}}$ where $\theta_{\max} = \max_{j \in \{1,...,d\}}\{\theta_j\}$ \cite{Yap_2011}.
The next theorem provides explicit conditions under which a time-delay embedding is stable and establishes explicit bounds on the quality of the embedding. 
First, define $A_1 = \sigma_{\min}(\bV)$ and $A_2 = \sigma_{\max}(\bV)$, where $\sigma(\cdot)$ denotes the singular value of a matrix. 
Let 
\begin{align}\label{eq:kappa_vals}
    \kappa_1 = \min_{j \in \{1,...,d\}} \frac{\|\bv_j^H \bh\|}{\|\bh\|_2} \qquad \kappa_2 = \max_{j \in \{1,...,d\}} \frac{\|\bv_j^H \bh\|}{\|\bh\|_2}
\end{align}
which identify the attractor dimensions that are ``least" and ``most" observable by the particular observable $\bh$. 
Finally, let 
\begin{align}
    \nu := \max_{p \neq q} \left\{\frac{1}{|\sin(\theta_p \tau)|}, \frac{1}{\sin\left(\frac{(\theta_p + \theta_q) \tau}{2} \right)}, \frac{1}{\sin\left(\frac{(\theta_p - \theta_q) \tau}{2} \right)} \right\},
\end{align}
a measure which quantifies the tendency of the embedding to introduce aliasing, a property that depends on the choice of delay $\tau$ and the system frequencies $\{\pm\theta_j\}_{j = 1}^d$. 
With the above terms, the second theorem proved by \cite{Yap_2011} is as follows.
\medskip

\begin{theorem} \label{thm:Stable_Takens_Embedding_2}
    (Stable Linear Takens' Embedding; Theorem III.2 \cite{Yap_2011}) Consider a class-$\cA(d)$ linear system in $\R^n$ in steady-state. 
    Let $\tau > 0$ be the sampling rate, $\bh \in \R^n$ be the observation function with $\|\bh\|_2^2 = \frac{2d}{\ell}$, and $\bPsi$ be the delay-coordinate map \eqref{eq:Uniform_Delay_Coordinate_Map} with $\ell$ delays. 
    Suppose that $\ell > \left((2d-1)\frac{A_2 \kappa_2^2}{A_1 \kappa_1^2} \nu \right)$, the $\cA_{\bPhi}$-eigenvalues $\{e^{\pm i \theta_j \tau}\}$ are distinct and imaginary, and $\bv_j^H \bh \neq 0$ for all $j \in \{1,...,d\}$. 
    Then, for all distinct pairs of points $\bx_1, \bx_2 \in \cM$, $\bPsi$ satisfies \eqref{eq:Stable_Embedding} for constants $C:= d \left(\frac{\kappa_1^2}{A_2} + \frac{\kappa_2^2}{A_1}\right)$ and $\delta := \delta_0 + \delta_1(\ell)$, where
    \begin{align}
        \delta_0 := \frac{A_2 \kappa_2^2 - A_1 \kappa_1^2}{A_2 \kappa_2^2 + A_1 \kappa_1^2}, \, \, \delta_1(\ell) = \frac{(2d-1) \nu}{\ell} \frac{2 A_2 \kappa_2^2}{A_2 \kappa_2^2 + A_1 \kappa_1^2}
    \end{align}
\end{theorem}
\bigskip

Theorem \ref{thm:Stable_Takens_Embedding_2} implies that the observation vector $\bh$ must be able to observe each of the $2d$ dimensions of the attractor on which solutions evolve, and that the system and sampling frequencies must be distinct to avoid aliasing.
The embedding approaches isometry--a distance-preserving mapping--when $\delta \to 0$.
As $\ell \to \infty$, $\delta = \delta_0 + \delta(\ell) \to \delta_0$, and for perfect isometry to occur, we must have $\delta \to 0$, and thus $\delta_0 \to 0$. 
This occurs provided the dynamics of the system attractor uniformly fill the state space (i.e. $A_1 \approx A_2$) and provided the observation function captures information on each of the dimensions uniformly (i.e. $\kappa_1 \approx \kappa_2$).
With the stable linear Takens' embedding theorems as references, we can extend the results to delay-coordinate maps using unevenly-spaced time-delays.

\section{Theoretical Results}
\subsection{Non-Uniform Delay-Coordinate Maps}
Given a set of $\ell$ times $\{\tau_k\}_{k = 1}^\ell$, define the non-uniformly-sampled delay-coordinate map $\bPsi_{\{\bh, \bPhi, \{\tau_k\}_{k = 1}^\ell\}}$ as:
\begin{align} \label{eq:Nonuniform_Delay_Coordinate_Map}
    \bPsi(\bx(t)) 
    = \begin{bmatrix}
        \by(t - \tau_1) \\ \by(t - \tau_2) \\ \vdots \\ \by(t - \tau_\ell)
    \end{bmatrix}
    = \begin{bmatrix}
        \bh^T \bPhi^{-\tau_1}\bx(t) \\ \bh^T \bPhi^{-\tau_2} \bx(t) \\ \vdots \\ \bh^T \bPhi^{-\tau_\ell} \bx(t)
    \end{bmatrix}
\end{align}
Consider a solution $\bx(t) = \bV \balpha(t)$ on the attractor $\cM$ as in Definition \ref{def:Attractor}. Component-wise, the embedding $\bPsi(\bx(t))$ is:
\begin{align*}
    (\bPsi(\bx(t)))_{k} 
    &= \bh^T(\bPhi^{\tau_k} \bx(t)) 
    = \bh^T(\bPhi^{\tau_k} \bV \balpha(t)) \\
    &= \bh^T(\bV \bD^{\tau_k} \balpha(t)) 
    = \bg_{\tau_k}^H \balpha(t)
\end{align*}
where $\bg_{\tau_k}^H := \bh^T \bV \bD^{\tau_k}$ expressed element-wise is:
\begin{align*}
    \bg_{\tau_k}^H &= \begin{bmatrix}
        \bv_1^T \bh e^{-i \theta_1 \tau_k}, \bv_1^H \bh e^{i \theta_1 \tau_k}, \cdots, \bv_d^T \bh e^{-i\theta_d \tau_k}, \bv_d^H \bh e^{i \theta_d \tau_k}
    \end{bmatrix}
\end{align*}
Concatenating the vectors $\bg_{\tau_k}^H$ for $k \in \{1,...,\ell\}$ gives rise to the following matrix $\bG$, which maps from the space of coefficients $\balpha \in \C^{2d}$ specifying solutions evolving on the attractor to the output delay-coordinate space in $\C^\ell$:

\begin{align}
    \bG = \begin{bmatrix}
        \bg_{\tau_1}^H \\
        \bg_{\tau_2}^H \\
        \vdots \\
        \bg_{\tau_\ell}^H
    \end{bmatrix} = \bT \bH \in \C^{\ell \times 2d}
\end{align}

For details about the derivation of $\bG$, see the proofs of Theorem III.1 and Theorem III.2 in \cite{Yap_2011}. 
$\bG$ can be partitioned into two matrices: a Vandermonde matrix $\bT$ capturing the temporal information, and a diagonal matrix $\bH$ containing the observed spatial information:
\begin{align*}
    \bT = \begin{bmatrix}
        e^{-i \theta_1 \tau_1} & e^{i \theta_1 \tau_1} & ... & e^{-i \theta_d \tau_1} & e^{i \theta_d \tau_1} \\
        e^{-i \theta_1 \tau_2} & e^{i \theta_1 \tau_2} & ... & e^{-i \theta_d \tau_2} & e^{i \theta_d \tau_2} \\
        \vdots & \vdots & & \vdots & \vdots \\
        e^{-i \theta_1 \tau_{\ell}} & e^{i \theta_1 \tau_{\ell}} & ... & e^{-i \theta_d \tau_{\ell}} & e^{i \theta_d \tau_{\ell}} \\
    \end{bmatrix} \in \C^{\ell \times 2d}
\end{align*}
\begin{align*}
    \bH = \diag(\bv_1^T \bh, \bv_1^H \bh, ..., \bv_d^T \bh, \bv_d^H \bh) \in \C^{2d \times 2d}
\end{align*}
Only $\bT$ is time-dependent, so in analyzing the effect of unevenly-spaced time-delays on the existence and stability of the time-delay embedding, we must consider generalized Vandermonde matrices $\bT$ that involve non-consecutive, non-integer powers $\{\tau_k\}_{k = 1}^\ell$, which we turn to in the next section.

\subsection{Generalized Vandermonde Matrices and Times Series}
Given points $\{z_p\}_{p = 1}^m \in \C$, the Vandermonde matrix $\bT \in \C^{m \times m}$ \cite{Golub_1996, Horn_Johnson_1985} is defined as:
\begin{align}\label{eq:Vandermonde_Matrix}
    \bT = \begin{bmatrix}
        z_1^0 & z_2^0 & \cdots & z_m^0 \\
        z_1^1 & z_2^1 & \cdots & z_m^1 \\
        \vdots & \vdots & \ddots & \vdots \\
        z_1^{m-1} & z_2^{m-1} & \cdots & z_m^{m-1} \\
    \end{bmatrix}
\end{align}
and its determinant is:
\begin{align} \label{eq:Vandermonde_Determinant}
    \det(\bT) = \prod_{1 \le p \le q}^m (z_p - z_q)
\end{align}
The rank of $\bT$ equals the number of distinct points in $\{z_p\}_{p = 1}^m$.
That is, when $z_p \neq z_q$ for all $p \neq q$, i.e. all $m$ points are distinct, $\det(\bT) \neq 0$, which implies that $\rank(\bT) = m$ and thus $\bT$, as a mapping, preserves all relevant information.
It is well-known that Vandermonde and Vandermonde-like matrices are ill-conditioned \cite{Pan_2016}, and its determinant can grow quite large as additional points are added.
If $\{z_p\}_{p = 1}^m$ are any set of $m$ points within a compact subset $\Omega \subset \C$, then the set of points that maximizes the Vandermonde determinant \eqref{eq:Vandermonde_Determinant}, known as the Fekete points \cite{Fekete1923}, gives rise to:
\begin{align}\label{eq:Finite_Diameter}
    d_m(\Omega) = \max_{\{z_1,...,z_m\} \in \Omega} \left(\prod_{1 \le p \le q}^m |z_p - z_q|\right)^{1/\binom{m}{2}}
\end{align}
Because the Vandermonde matrix is ill-conditioned, the exponent $1/\binom{m}{2}$, which is the number of pairwise products between the $m$ points, is used as a normalization.
Normalization ensures that the limit exists as $m \to \infty$ in \eqref{eq:Finite_Diameter}.
This limit is known as the transfinite diameter of $\Omega$ \cite{Kirsch_2005}:
\begin{align}\label{eq:Transfinite_Diameter}
    \tau(\Omega) = \lim_{m \to \infty}(d_m(\Omega))
\end{align}
While explicit results for Fekete points for general $\Omega$ are often difficult to determine explicitly, when $\Omega$ is the unit circle in the complex plane $S^1 := \{z : |z| = 1\}$, the $m$ Fekete points are the $m$th roots of unity, and $\tau(\Omega) = 1$.

The Vandermonde matrices we consider for time series involve points $\{z_p = e^{\pm i \theta_p}\}_{p=1}^m \subseteq S^1$. 
These types of Vandermonde matrices are known as Fourier Vandermonde matrices and are well-studied in the case for when the time sampling rates are evenly-spaced \cite{Ryan_2009, Tucci_2012, Batenkov_2018, Aubel_2019, Batenkov_2021, Kunis_2021}.
For these matrices, the set of system frequencies is generally not uniformly-distributed on the unit circle, so a given set of system frequencies is unlikely to maximize the Vandermonde determinant \eqref{eq:Vandermonde_Determinant}. 
Furthermore, the above statements apply only to Vandermonde matrices that are square, which only corresponds to delay-coordinate maps where the number of time delays used to construct the map $\ell$ is exactly equal to the underlying dimension of the system attractor $2d$, and matrices that are constructed using consecutive integer powers, which implies the time series data must not be unevenly-sampled.
What is helpful from \eqref{eq:Finite_Diameter} and $\eqref{eq:Transfinite_Diameter}$ is the $1/\binom{m}{2}$ scaling factor for the determinant of the Vandermonde matrix that ensures the determinant converges to $1$ as $m \to \infty$.

To generalize \eqref{eq:Finite_Diameter} and \eqref{eq:Transfinite_Diameter} to settings where more lags than attractor dimensions are used in the delay-coordinate map (i.e. $\ell \ge 2d$) and when the map is constructed of non-uniform lags, we use a set of coordinates $\{z_p\}_{p = 1}^m$ and powers $\{\tau_k\}_{k = 1}^\ell$ to define a generalized Vandermonde matrix, similar to that studied in \cite{Ryan_2009, Nagel_2020}:
\begin{align}\label{eq:Generalized_Vandermonde_Matrix}
    \bT = \bT_{m,\ell} = \begin{bmatrix}
        z_1^{\tau_1} & z_2^{\tau_1} & \cdots & z_m^{\tau_1} \\
        z_1^{\tau_2} & z_2^{\tau_2} & \cdots & z_m^{\tau_2} \\
        \vdots & \vdots & \ddots & \vdots \\
        z_1^{\tau_\ell} & z_2^{\tau_\ell} & \cdots & z_m^{\tau_\ell} \\
    \end{bmatrix} \in \C^{\ell \times m}
\end{align}

The typical approach to prove that Vandermonde matrices $\bT$ as in \eqref{eq:Vandermonde_Matrix} are full rank involves having an explicit expression for the determinant \eqref{eq:Vandermonde_Determinant} and inferring the condition that the $m$ nodes $\{z_p\}_{p = 1}^m$ must be distinct.
Deriving the determinant depends on properties of polynomials of degrees $\{1,...,m\}$.
Consequently, because the generalized Vandermonde matrix $\bT$ as in \eqref{eq:Generalized_Vandermonde_Matrix} involves unevenly-spaced, non-integer powers, the original argument does not readily extend to the generalized setting.
Thus, we state the following, only as a conjecture, about the rank of a generalized Vandermonde matrix.
\medskip

\begin{conjecture} \label{conj:Gen_Vand_Rank}
    Let $\bT_{m,\ell} \in \C^{\ell \times m}$ be as \eqref{eq:Generalized_Vandermonde_Matrix}. 
    Let $\{\tau_k\}_{k = 1}^\ell$ be a set of strictly-positive, distinct powers in the sense that $\tau_a \neq \tau_b$ for any $a,b \in \{1,...,\ell\}$, and let $\{z_p\}_{p = 1}^m \subseteq S^1$ be a distinct set points with associated angles $\{\theta_p\}_{p = 1}^m$ in the interval $(-\pi,\pi]$ in the sense that $\theta_p \neq \theta_q$ for $p,q \in \{1,...,m\}$. 
    Then with probability $1$, $\rank(\bT_{m,\ell}) = \min\{m,\ell\}$.
\end{conjecture}
\bigskip

\begin{IEEEproof}
    The conjecture is trivially true when $m = 1$ and $\ell = 1$ as the matrix $\bT_{1,1} = [z_1^{\tau_1}]$ is non-singular provided $z_1 \neq 0$, a condition which is valid when assuming $z_1 \in S^1$.
    When extending the number of time samples used, so that we consider $\bT_{1,\ell} = [z_1^{\tau_1},...,z_1^{\tau_\ell}]^T \in \C^{\ell \times 1}$, it follows that $\rank(\bT_{1,\ell}) = 1$ for all $\ell \in \mathbb{N}$ because the sub-matrix consisting of the first row is always non-trivial.
    \smallskip

    The case when $m = 2$ and $\ell \ge 2$ is also relatively straightforward to establish. 
    For $m = 2$ and $\ell = 2$:
    \begin{align}
        \bT_{2,2} =
        \begin{bmatrix}
            z_1^{\tau_1} & z_2^{\tau_1} \\
            z_1^{\tau_2} & z_2^{\tau_2}
        \end{bmatrix}
    \end{align}
    the determinant of which is:
    \begin{align}
        \det(\bT_{2,2}) = z_1^{\tau_1} z_2^{\tau_1} (z_2^{\tau_2 - \tau_1} - z_1^{\tau_2 - \tau_1})
    \end{align}
    Letting $z_1 = e^{i \theta_1}$ and $z_2 = e^{i \theta_2}$ with $\theta_1, \theta_2 \in (-\pi,\pi]$ and $\delta_1 := \tau_2 - \tau_1$, then, to ensure $\bT_{2,2}$ is non-singular by showing $\det(\bT_{2,2}) \neq 0$, we require that $z_1, z_2 \neq 0$, which is trivially true when $z_1, z_2 \in S^1,$ and also that $\theta_2 - \theta_1 \neq 2\pi \left(k_1 + \frac{k_2}{\delta_1}\right)$ for $k_1, k_2 \in \mathbb{Z}$--a condition accounting for the multi-valued nature of the functions.
    This condition for unevenly-spaced time delay samples for generalized Vandermonde matrix is analogous to the condition for Fourier Vandermonde matrices with evenly-spaced delay powers that requires $\theta_2 - \theta_1 \neq 2 \pi k$ for $k \in \mathbb{Z}$, albeit with slightly more restrictions.
    
    If $\delta_1$ is a rational number expressed in lowest terms--that is, $\delta_1 = p/q$ where $p, q \ge 1$ are co-prime--then the condition reduces to $\theta_2 - \theta_1 \neq \frac{2\pi k}{p}$ for $k \in \mathbb{Z}$. 
    Note that for $\theta_1, \theta_2 \in (-\pi,\pi]$, the interval of interest for $\theta_2 - \theta_1$ is then $(-2\pi, 2\pi]$, meaning increasing $p$ begins to increase the number of values that can introduce aliasing on $(-2\pi,2 \pi]$, suggesting that non-uniformly-spaced powers can actually be more prone to aliasing than using evenly-spaced powers.
    
    This is further confirmed if $\delta_1$ is irrational.
    Since $\mathbb{Z}$ and $\frac{1}{\delta_1} \mathbb{Z}$ are linearly independent over the rational numbers, for irrational $\delta_1$ the set $2\pi \left(k_1 + \frac{k_2}{\delta_1}\right)$ is Lebesgue measure zero but dense on $(-2\pi,2\pi]$, meaning that non-uniform time sampling that have irrational spacing may be even more prone to aliasing than when using spacings involving rational numbers.
    Notably, while there is a dense set of choices of $\theta_1$ and $\theta_2$ for which $\bT_{2,2}$ will fail to be full rank, since the set is Lebesgue measure zero generic choices of $\theta_1$ and $\theta_2$ will ensure that $\bT_{2,2}$ is full rank with probability one.
    
    For $\bT_{2,\ell}$ with $\ell > 2$ to be non-singular, a necessary and sufficient condition is to require at least one $\bT_{2,2}$ sub-matrix consisting of two rows of $\bT_{2,\ell}$ to be non-singular: that is, $\bT_{2,\ell}$ must have at least two rows that are linearly independent.
    We could form such a $\bT_{2,2}$ sub-matrix using any choice of two rows of $\bT_{2,\ell}$, say with rows raised to time powers $\tau_a$ and $\tau_b$, where $a,b \in \{1,...,\ell\}$.
    The condition for $\rank(\bT_{2,\ell}) = 2$ would then be that there exists at least one set of time pairings $\tau_a, \tau_b$ such that $\theta_2 - \theta_1 \neq 2 \pi \left(k_1 + \frac{k_2}{\tau_a - \tau_b}\right)$.
    Since only one of the $\binom{\ell}{2}$ possible row pairings of $\bT_{2,\ell}$ needs to satisfy the above condition, we need
    \begin{align*}
        \frac{\theta_2 - \theta_1}{2\pi} \in (-1,1)\backslash\{0\} \cup \bigcap_{1 \le b < a \le \ell} \left\{\alpha \in \frac{1}{\delta_{a,b}} \mathbb{Z}: \alpha \in (-1,1]\right\}
    \end{align*}
    As $\ell$ increases, the intersection of the sets of aliased points between the $\binom{\ell}{2}$ pairs of time samples becomes smaller and smaller; thus, increasing lags can mitigate aliasing.
    \smallskip
    
    To show the compounding difficulty of proving the statement for $\ell \ge 3$, we also consider $\ell = 3$. 
    Let $z_1 = e^{i \theta_1}, z_2 = e^{i \theta_2},$ and $z_3 = e^{i \theta_3} \in S^1$, with $\theta_1, \theta_2, \theta_3 \in (-\pi,\pi]$. 
    Then,
    \begin{align}
        \bT_{3,3} = 
        \begin{bmatrix}
            z_1^{\tau_1} & z_2^{\tau_1} & z_3^{\tau_1} \\
            z_1^{\tau_2} & z_2^{\tau_2} & z_3^{\tau_2} \\
            z_1^{\tau_3} & z_2^{\tau_3} & z_3^{\tau_3} \\
        \end{bmatrix}
    \end{align}
    Letting $\delta_1 := \tau_2 - \tau_1$ and $\delta_2 := \tau_3 - \tau_2$, the determinant of $\bT_{3,3}$ can be computed first by subtracting $z_1^{\delta_1}$ times row $1$ from row $2$ and by subtracting $z_1^{\delta_2}$ times row $2$ from row $3$:
    \begin{align*}
        \det(\bT_{3,3}) &= 
        \begin{vmatrix}
            z_1^{\tau_1} & z_2^{\tau_1} & z_3^{\tau_1} \\
            z_1^{\tau_2} - z_1^{\delta_1} z_1^{\tau_1} & z_2^{\tau_2} - z_1^{\delta_1} z_2^{\tau_1} & z_3^{\tau_2} - z_1^{\delta_1} z_3^{\tau_1} \\
            z_1^{\tau_3} - z_1^{\delta_2} z_1^{\tau_2} & z_2^{\tau_3} - z_1^{\delta_2} z_2^{\tau_2} & z_3^{\tau_3} - z_1^{\delta_2} z_3^{\tau_2}
        \end{vmatrix}
    \end{align*}
    By simplifying and expanding by the first column, and computing the resulting $2\times2$ determinant, we have the following:
    \begin{align}\label{eq:T_33}
    \begin{split}
        \det(\bT_{3,3}) &= 
        \begin{vmatrix}
            z_1^{\tau_1} & z_2^{\tau_1} & z_3^{\tau_1} \\
            0 & z_2^{\tau_1}(z_2^{\delta_1} - z_1^{\delta_1}) & z_3^{\tau_1}(z_3^{\delta_1} - z_1^{\delta_1}) \\
            0 & z_2^{\tau_2}(z_2^{\delta_2} - z_1^{\delta_2}) & z_3^{\tau_2}(z_3^{\delta_2} - z_1^{\delta_2})
        \end{vmatrix} \\
        &= z_1^{\tau_1} \cdot \begin{vmatrix}
            z_2^{\tau_1}(z_2^{\delta_1} - z_1^{\delta_1}) & z_3^{\tau_1}(z_3^{\delta_1} - z_1^{\delta_1}) \\
            z_2^{\tau_2}(z_2^{\delta_2} - z_1^{\delta_2}) & z_3^{\tau_2}(z_3^{\delta_2} - z_1^{\delta_2})
        \end{vmatrix} \\ 
        &= z_1^{\tau_1} z_2^{\tau_1} z_3^{\tau_1} (z_2^{\delta_1} - z_1^{\delta_1}) (z_3^{\delta_1} - z_1^{\delta_1}) \\
        &\quad \cdot \left(z_3^{\delta_1} \frac{(z_3^{\delta_2} - z_1^{\delta_2})}{(z_3^{\delta_1} - z_1^{\delta_1})} - z_2^{\delta_1} \frac{(z_2^{\delta_2} - z_1^{\delta_2})}{(z_2^{\delta_1} - z_1^{\delta_1})} \right) \\
    \end{split}
    \end{align}
    For $\bT_{3,3}$ to be non-singular, we require $\det(\bT_{3,3}) \neq 0$.
    Similar to $\bT_{2,2}$, necessarily $z_1, z_2, z_3 \neq 0$, which is true since $z_1, z_2, z_3 \in S^1$.
    Furthermore, the relationships between $z_1, z_2$ and $z_3$ will impose restrictions.
    When working with the non-uniform delays, the multi-valued nature of functions $z^\delta$, where $\delta$ is potentially non-integer, yields sufficient conditions: $\theta_1 - \theta_2 \neq 2\pi (k_1 + \frac{k_2}{\delta_1})$ and $\theta_1 - \theta_3 \neq 2 \pi (k_1 + \frac{k_2}{\delta_1})$ for $k_1,k_2 \in \mathbb{Z}$.
    Like the case of $\bT_{2,2}$, if $\delta_1$ is a rational number in lowest terms $\delta_1 = p/q$ so that $p, q \ge 1$ are co-prime, then the condition reduces to $\theta_1 - \theta_2 \neq \frac{2\pi k}{p}$ and $\theta_2 - \theta_3 \neq \frac{2\pi k}{p}$ for $k \in \mathbb{Z}$. 
    If $\delta_1$ is irrational, then there are a dense set of values in the interval $(-2\pi, 2\pi]$ for which the differences of frequencies lead to aliasing.
    
    The last requirement for $\bT_{3,3}$ to be non-singular demands a more complicated relationship between $z_2$ and $z_3$ as seen in \eqref{eq:T_33}. 
    We can definitely say that if $z_2 = z_3$, then the final term is zero.
    Thus, for the last term to be nonzero, $\theta_2 - \theta_3 \neq 2 \pi k$ for $k \in \mathbb{Z}$ for $k \in \mathbb{Z}$ is still necessary, but not sufficient.
    Regardless, the points at which aliasing could occur are still countably many and thus represent a measure zero set of possible choices, meaning that generically aliasing should not be a significant issue.
    And, like the the case of $\bT_{2,\ell}$, when extending to general $\bT_{3,\ell}$, increasing the lags can provide more diverse information that addresses the aliasing problem.
    
    For $m > 3$, the determinant of \eqref{eq:Generalized_Vandermonde_Matrix} only becomes more complicated since the non-uniform sampling adds more conditions involving rational functions in combinatorial fashion.
\end{IEEEproof}
\medskip

Although the above conjecture may be difficult to prove for larger $m$ and $\ell$ values, we will show that it is possible to derive explicit analytical results for an upper bound--but unfortunately not for the more important lower bound--on the scaled determinant of $\bT$ \eqref{eq:Generalized_Vandermonde_Matrix} for all $\ell$, which will at least ensure we avoid $\bT$ being non-singular from having unbounded singular values.
In addition, we can establish positive results in the asymptotic regime as $\ell \to \infty$ for the scaled determinant of $\bT$, its condition number and its rank.
Finally, we perform numerical simulations to generate empirical results as another means of supporting the conjecture.
\medskip

\begin{definition}\label{def:Pseudo_Transfinite_Diameter}
    Let $\bT \in \C^{\ell \times m}$ be a generalized Vandermonde matrix as in \eqref{eq:Generalized_Vandermonde_Matrix}. We say the \textit{pseudo-transfinite diameter} is:
    \begin{align}\label{eq:Pseudo_Finite_Diameter}
        \tilde{d}_m(\bT) = \left(\det\left(\left(\sqrt{\frac{m}{\ell}} \bT^H\right) \left(\sqrt{\frac{m}{\ell}} \right) \bT \right)^{1/2}\right)^{1/\binom{m}{2}}
    \end{align}
    If we define $\bS = \left(\sqrt{\frac{m}{\ell}} \bT^H\right)\left(\sqrt{\frac{m}{\ell}} \bT\right)$, then $\tilde{d}_m(\bT) = \left(\det(\bS)^{1/2}\right)^{1/\binom{m}{2}}$.
\end{definition}
\bigskip

The intuition of why $\sqrt{m/\ell}$ is an appropriate scaling factor for $\bT$ in the above definition relates to the total ``energy" of the matrix, which is connected to the Frobenius norm.
The Frobenius norm of a matrix $\bT \in \C^{\ell \times m}$ with entries $T_{pk}$ is related to the sum of its squared entries and the trace of $\bT^H \bT$: 
\begin{align}
    \|\bT\|_F^2 = \sum_{p = 1}^m \sum_{k = 1}^\ell T_{pk}^2 = \tr(\bT^H \bT)
\end{align}
The Frobenius norm of $\sqrt{m/\ell} \bT$ is then: 
\begin{align}
    \left\|\sqrt{\frac{m}{\ell}} \bT \right\|_F^2 = \sum_{p = 1}^m \sum_{k = 1}^\ell \frac{m}{\ell} T_{pk}^2 = \sum_{p = 1}^m m \left(\frac{1}{\ell}\sum_{k = 1}^\ell T_{pk}^2\right)
\end{align}
In effect, the scaling factor $\sqrt{m/\ell}$ averages the total energy contribution of the $T_{pk}^2$ terms over the number of lags $\ell$ used and rescales that average by $m$ to be comparable to the behavior of the case when $\ell = m$.

To analyze $\bT$, we have the following lemma about the entries of $\bS$.
\medskip

\begin{lemma}\label{lem:A_Gerschgorin}
    Let $\bS = (m/\ell) \bT^H \bT$ with $\bT$ as in \eqref{eq:Generalized_Vandermonde_Matrix} and entries $s_{pq}$ with $p,q\in \{1,...,m\}$.
    Then the entries of $\bS$ are:
    \begin{align}
        s_{pp} = m \qquad s_{pq} = \frac{m}{\ell} \sum_{k = 1}^\ell e^{-i(\theta_p - \theta_q) \tau_k}
    \end{align}
    and the radii of the Gerschgorin discs of $\bS$ are:
    \begin{align}\label{eq:Gerschgorin_Rad}
        R_p := \sum_{q = 1, p \neq q}^m |s_{pq}| = \sum_{q = 1, p \neq q}^m \left|\frac{m}{\ell} \sum_{k = 1}^\ell e^{-i(\theta_p - \theta_q)\tau_k} \right| 
    \end{align}
\end{lemma}
\bigskip

\begin{IEEEproof}
    For $\bS = (m/\ell) \bT^H \bT$ with $\bT$ as in \eqref{eq:Generalized_Vandermonde_Matrix}, the diagonal entries of $\bS$, denoted by $s_{pp}$, are:
    \begin{align}
        s_{pp} = \frac{m}{\ell} \sum_{k = 1}^\ell \bar{z}_p^{\tau_k} z_p^{\tau_k} = \frac{m}{\ell}  \sum_{k = 1}^\ell \left(\bar{z}_p z_p \right)^{\tau_k}
    \end{align}
    When $z_p = e^{i \theta_p}$, then $\bar{z}_p = e^{-i \theta_p}$, $\bar{z}_p z_p = 1$, in which case $s_{pp} = (m/\ell) \cdot \ell = m$. The off-diagonal entries of $\bS$ are:
    \begin{align}
        s_{pq} = \frac{m}{\ell} \sum_{k = 1}^\ell \bar{z}_p^{\tau_k} z_q^{\tau_k} = \frac{m}{\ell} \sum_{k = 1}^\ell e^{-i(\theta_p - \theta_q) \tau_k}
    \end{align}
    Summing over the $q$s in each row excluding the diagonal terms yields the expression for $R_p$.
\end{IEEEproof}
\medskip

We prove in the next theorem that the pseudo-transfinite diameter as in Definition \ref{def:Pseudo_Transfinite_Diameter} is bounded above and that the bound converges to $1$ at a sub-linear rate based on the number of points $m$. 
This ensures the generalized determinant of $\bT$ is bounded and that $\bT$ cannot be non-singular due to unbounded singular values.
The result applies for any choice of lag $\ell$.
\medskip

\begin{theorem}\label{thm:Gen_Vand_Mat_Det_Bound}
    Let $\bT \in \C^{\ell \times m}$ be as in \eqref{eq:Generalized_Vandermonde_Matrix}. 
    Consider points $z_p = e^{i \theta_p}$ for $p \in \{1,...,m\}$ where the $\{\theta_p\}_{p = 1}^m \in \R$ are distinct.
    Let $\{\tau_k\}_{k = 1}^\ell > 0$ be a set of distinct powers.
    Then, 
    \begin{align}
        \tilde{d}_m(\bT) \le m^{1/(m-1)}
    \end{align}
    and $\tilde{\tau}(\bT) := \lim_{m \to \infty} \tilde{d}_m(\bT) = 1$.
\end{theorem}
\bigskip

\begin{IEEEproof}
    For a given $\bT \in \C^{\ell \times m}$, let $\bS = \left(\sqrt{\frac{m}{\ell}} \bT^H\right)\left(\sqrt{\frac{m}{\ell}} \bT\right)$.
    By Lemma \ref{lem:A_Gerschgorin}, $s_{pp} = m$.
    Then:
    \begin{align*}
        \tr(\bS) = \sum_{p = 1}^m s_{pp} = \sum_{p = 1}^m m = m^2
    \end{align*}
    Since $\bS = (m/\ell) \bT^H \bT$, $\tr(\bS) = \left(m/\ell\right) \tr(\bT^H \bT)$, which, when combined with the above expression, yields:
    \begin{align}\label{eq:Gen_Vand_Tr}
        \tr(\bT^H \bT) = m \ell
    \end{align}
    Next, we consider $\tilde{d}_m(\bT)$, beginning with $\det(\bS)^{1/2}$. 
    (Note that when $m = \ell$, $\det(\bS)^{1/2} = \det(\bT)$, as expected, and, when $\bT$ is as in \eqref{eq:Vandermonde_Matrix}, \eqref{eq:Pseudo_Finite_Diameter} is equivalent to \eqref{eq:Finite_Diameter} as desired.) 
    For the generalized case where $m \neq \ell$ and $\bT$ is as in \eqref{eq:Generalized_Vandermonde_Matrix}:
    \begin{align} \label{eq:Gen_Vand_Det_1}
    \begin{split}
        \det(\bS)^{1/2} &= \det \left(\left(\frac{m}{\ell}\right) \bT^H \bT \right)^{1/2} \\
        &= \left(\frac{m}{\ell}\right)^{m/2} \left(\prod_{p = 1}^m \lambda_p(\bT^H \bT)\right)^{1/2} \\
        &= \left(\frac{m}{\ell}\right)^{m/2} \prod_{p = 1}^m \sigma_p(\bT)
    \end{split}
    \end{align}
    By the Geometric Mean - Quadratic Mean inequality, applicable to a set of real numbers greater than 0, which is certainly true for the set of singular values $\{\sigma_p(\bT)\}_{p = 1}^m$:
    \begin{align} \label{eq:Gen_Vand_Det_GM_QM}
        \prod_{p = 1}^m \sigma_p(\bT) \le \left(\frac{1}{m}\sum_{p = 1}^m \sigma_p(\bT)^2\right)^{m/2}
    \end{align}
    Combining \eqref{eq:Gen_Vand_Det_1} and \eqref{eq:Gen_Vand_Det_GM_QM}:
    \begin{align*}
        \tilde{d}_m(\bT) \le \left(\frac{m}{\ell} \times \frac{1}{m}\sum_{p = 1}^m \sigma_p(\bT)^2\right)^{1/(m-1)}
    \end{align*}
    Since $\sum_{p = 1}^m \sigma_p(\bT)^2 = \tr(\bT^H \bT)$,
    \begin{align*}
        \tilde{d}_m(\bT) \le \left(\frac{1}{\ell} \tr(\bT^H \bT)\right)^{1/(m-1)}
    \end{align*}
    With $\tr(\bT^H \bT) = m \ell$ from \eqref{eq:Gen_Vand_Tr},
    \begin{align*}
        \tilde{d}_m(\bT) \le m^{1/(m-1)}
    \end{align*}
    As $m \to \infty$, $m^{1/(m-1)} \to 1$, meaning the pseudo-transfinite diameter is bounded above by $1$ in the limit, and the rate of convergence is at least related to $m^{1/(m-1)}$, which is $O\left(\frac{\ln(m)}{m}\right)$.
    While the bound on the convergence rate is sub-linear in terms of $m$--the number of frequencies or dimensions present in the system, importantly the result does converge and it does so independently of the number of lags $\ell$ used. 
    The only inequality in the proof was in \eqref{eq:Gen_Vand_Det_GM_QM}, so any improvement on the bound would require a tighter inequality.
\end{IEEEproof}
\medskip

\begin{remark}
    If $\bT$ was normalized not by $\sqrt{m/\ell}$ but instead by $1/\sqrt{\ell}$, the above bound would be $1$ for all $m > 1$. 
    This would provide a uniform upper bound and would match the Fourier Vandermonde matrices considered by \cite{Ryan_2009}, but may not scale the ``energy" to be comparable to a regular square Vandermonde matrix.
\end{remark}
\bigskip

Let the delays $\{\tau_k\}_{k = 1}^\ell$ be independent and identically-distributed (i.i.d.) random variables, and suppose, for consistency with the evenly-sampled setting, that the interval of interest for the particular distribution of each $\tau_k$ for $k \in \{1,...,\ell\}$ is mostly contained in $[0,\ell]$. 
To analyze the behavior of $R_p$ \eqref{eq:Gerschgorin_Rad} as $\ell \to \infty$, we first define an empirical measure:
\begin{align}
    \mu_\ell(\tau) := \frac{1}{\ell} \sum_{k = 1}^\ell \delta_{\tau_k}(\tau)
\end{align}
where $\delta$ here is the delta function. In this case:
\begin{align*}
    \frac{1}{\ell}\sum_{k = 1}^\ell e^{-i (\theta_p - \theta_q) \tau_k} = \int e^{-i(\theta_p - \theta_q) \tau} \dd \mu_\ell(\tau)
\end{align*}
For a random variable $\tau$, define the function $g(\tau):= e^{-i(\theta_p - \theta_q) \tau}$.
The Strong Law of Large Numbers \cite{Chung_2000} states that for every bounded measurable function $g$,
\begin{align}
    \int g(\tau) \dd \mu_\ell(\tau) = \frac{1}{\ell} \sum_{k = 1}^\ell g(\tau_k) \overset{a.s.}{\longrightarrow} \int g(\tau) \dd \mu(\tau)
\end{align}
meaning the empirical measures $\mu_\ell$ converge almost surely in the weak sense to the measure $\mu$, defined on the domain $\Omega$.
Since $g(\tau) = e^{-i(\theta_p - \theta_q) \tau}$ is a bounded measurable function,
\begin{align}\label{eq:LLN}
    \int e^{-i(\theta_p - \theta_q) \tau} \dd \mu_\ell(\tau) \overset{a.s.}{\longrightarrow} \int e^{-i(\theta_p - \theta_q) \tau} \dd \mu(\tau)
\end{align}
For convenience, define $\gamma_{pq} := \theta_p - \theta_q$. 
Then, considering the Gerschgorin radii of $\bS = \frac{m}{\ell} \bT^H \bT$ from \eqref{eq:Gerschgorin_Rad} and using \eqref{eq:LLN}:
\begin{align}
\begin{split}
    R_{p,\infty} &:= \lim_{\ell \to \infty} R_p \\
    &= \lim_{\ell \to \infty} m \sum_{q = 1, p \neq q}^m \left| \frac{1}{\ell} \sum_{k = 1}^\ell e^{-i \gamma_{pq} \tau_k} \right| \\
    &= m \sum_{q = 1, p \neq q}^m \left| \lim_{\ell \to \infty}  \frac{1}{\ell} \sum_{k = 1}^\ell e^{-i \gamma_{pq} \tau_k} \right| \\
    &= m \sum_{q = 1, p \neq q}^m \left| \int_{\Omega} e^{-i \gamma_{pq} \tau} \dd \mu(\tau) \right|
\end{split}
\end{align}

If the random variable $\tau$ is distributed according to the probability density function (PDF) $f(\tau)$ on a set $\Omega$, then
\begin{align} \label{eq:R_p_inf}
\begin{split}
    R_{p,\infty} &= m \sum_{q = 1, p \neq q}^m \left| \int_\Omega e^{-i \gamma_{pq} \tau} f(\tau) \dd \tau \right| \\
    &= m \sum_{q = 1, p \neq q}^m \left| \mathbb{E}[e^{-i \gamma_{pq} \tau}]\right|
\end{split}
\end{align}

By Gerschgorin's circle theorem (Theorem \ref{thm:Gerschgorin}), it follows that all the eigenvalues of $\bS$ are contained in the interval:
\begin{align}\label{eq:A_Gerschgorin_Spectrum}
    \lambda(\bS) \in \left[\min_{p \in \{1,...,m\}} \{s_{pp} - R_{p,\infty}\}, \max_{p \in \{1,...,m\}} \{s_{pp} + R_{p,\infty} \}\right]
\end{align}
With $\min_{p \in \{1,...,m\}} \{s_{pp}\} = m$, $\max_{p \in \{1,...,m\}} \{s_{pp}\} = m$, 
\begin{align}
    \lambda_{\min}(\bS) &\ge m \left(1 - \max_{p \in \{1,...,m\}} \sum_{q = 1, p \neq q}^m \left| \mathbb{E}[e^{-i \gamma_{pq} \tau}]\right| \right) \label{eq:Lam_S_min} \\ 
    \lambda_{\max}(\bS) &\le m \left(1 + \max_{p \in \{1,...,m\}} \sum_{q = 1, p \neq q}^m \left| \mathbb{E}[e^{-i \gamma_{pq} \tau}]\right| \right) \label{eq:Lam_S_max}
\end{align}
\medskip

\begin{theorem}\label{thm:Gen_Van_Mat_Det_Gerschgorin}
    Let $\bT \in \C^{\ell \times m}$ be a generalized Vandermonde matrix as in \eqref{eq:Generalized_Vandermonde_Matrix}. 
    Suppose that $\max_{p \in \{1,...m\}} R_{p,\infty} = 0$ with $R_{p,\infty}$ as in \eqref{eq:R_p_inf}. 
    Then,
    \begin{enumerate}
        \item $\lim_{\ell \to \infty} \rank(\bT) = m$;
        \item $\lim_{\ell \to \infty} \tilde{d}_m(\bT) = m^{1/(m-1)}$;
        \item $\lim_{\ell \to \infty} \kappa \left(\sqrt{\frac{m}{\ell}} \bT \right) = 1$.
    \end{enumerate}
\end{theorem}
\bigskip

\begin{IEEEproof}
    If $\max_{p \in \{1,...,m\}} R_{p,\infty} = 0$, then by \eqref{eq:Lam_S_min}, $\lambda_{\min}(\bS) \ge m$, and by \eqref{eq:Lam_S_max}, $\lambda_{\max}(\bS) \le m$. 
    So, all $m$ eigenvalues of $\bS$ converge to $m$. 
    Since $0 < m < + \infty$, all $m$ eigenvalues are nonzero, so $\rank(\bS) = m$, and, since $\rank(\bT) = \rank(\bT^H \bT) = \rank(\bS)$, $\rank(\bT) = m$.
    Next, the determinant is the product of the eigenvalues, so $\det(\bS) = m^m$, and, using \eqref{def:Pseudo_Transfinite_Diameter}, 
    \begin{align*}
        \tilde{d}_m(\bT) 
        = \left(\det(\bS)^{1/2}\right)^{1/\binom{m}{2}} 
        = m^{1/(m-1)}
    \end{align*} 
    Thus, in the asymptotic regime as $\ell \to \infty$, the upper bound on the determinant derived from Theorem \ref{thm:Gen_Vand_Mat_Det_Bound} is sharp.
    Finally, 
    \begin{align}
    \begin{split}
        \lim_{\ell \to \infty} \kappa \left(\sqrt{\frac{m}{\ell}} \bT \right) 
        = \frac{\sigma_{\max} \left(\sqrt{\frac{m}{\ell}} \bT \right)}{\sigma_{\min}\left(\sqrt{\frac{m}{\ell}} \bT \right)} 
        = \left(\frac{\lambda_{\max} (\bS)}{\lambda_{\min} (\bS)}\right)^{1/2} 
        = 1
        \end{split}
    \end{align}
\end{IEEEproof}
\bigskip

\begin{theorem}\label{thm:Gen_Vand_Mat_Cond_Bound}
    Let $\bT \in \C^{\ell \times m}$ be as in \eqref{eq:Generalized_Vandermonde_Matrix}. 
    Consider $z_p = e^{i \theta_p}$ for $p \in \{1,...,m\}$ with $\{\theta_p\}_{p = 1}^m \in \R$ distinct.
    Let $\{\tau_k\}_{k = 1}^\ell > 0$ be a set of distinct powers.
    Then, provided 
    \begin{align}\label{eq:Cond_Eps_Val}
        \max_{p \in \{1,...,m\}} \sum_{q = 1, p \neq q}^m \left| \mathbb{E}[e^{-i \gamma_{pq} \tau}]\right| \le \varepsilon < 1
    \end{align}
    the scaled condition number has the upper bound:
    \begin{align}\label{eq:Gen_Vand_Mat_Cond_Bound}
        \tilde{\kappa}_m \left(\sqrt{\frac{m}{\ell}} \bT \right) := \kappa_m \left(\sqrt{\frac{m}{\ell}} \bT \right)^{1/\binom{m}{2}} \le \left(\frac{1 + \varepsilon}{1 - \varepsilon}\right)^{\frac{1}{m(m-1)}}
    \end{align}
\end{theorem}
\bigskip

\begin{IEEEproof}
    Provided \eqref{eq:Cond_Eps_Val} is satisfied, then from \eqref{eq:Lam_S_max} $\lambda_{\max}(\bS) \le m(1 + \varepsilon)$ and from \eqref{eq:Lam_S_min} $\lambda_{\min} \ge m(1- \varepsilon)$. 
    Since $\kappa(\bS) = \left(\frac{\lambda_{\max}(\bS)}{\lambda_{\min}(\bS)}\right)^{1/2}$, by normalizing with the power $1/\binom{m}{2}$, the result follows.
\end{IEEEproof}
\medskip

The conclusions of Theorem \ref{thm:Gen_Van_Mat_Det_Gerschgorin} rest on the assumption that $\max_{p \in \{1,...,m\}}R_{p,\infty} = 0$. 
To show this is reasonable, we compute $R_{p,\infty}$ for different time-sampling schemes: (1) the uniform distribution, to simulate time-delays that are on average uniformly-spaced and which we would expect to behave like the deterministic evenly-spaced sampling scheme; (2) the exponential distribution, to stimulate delays across different orders of magnitude, potentially like in multi-scale modeling; and (3) the normal distribution.
\smallskip

\subsubsection{Uniform Distribution}
Suppose the delays $\{\tau_k\}_{k = 1}^\ell$ are i.i.d. uniform random variables, $\tau \sim \mathcal{U}[0, \ell]$ with $\Omega := [0, \ell]$. 
\begin{align}
    \mathbb{E}[e^{-i \gamma_{pq} \tau}] &= \int_\Omega e^{-i \gamma_{pq} \tau} f(\tau) \dd \tau 
    = \frac{1}{\ell} \int_{0}^{\ell} e^{-i \gamma_{pq} \tau} \dd \tau \nonumber \\
    &= \frac{1}{\ell} \frac{i}{ \gamma_{pq}} e^{-i \gamma_{pq} \ell/2} \left(e^{-i \gamma_{pq} \ell/2} -e^{i \gamma_{pq} \ell/2}\right) \nonumber \\
    &= \frac{2}{\ell \gamma_{pq}} e^{-i \gamma_{pq} \ell/2} \sin\left( \frac{\gamma_{pq} \ell}{2} \right)
\end{align}
Taking the modulus of the expected value nullifies the exponential oscillatory behavior, leaving:
\begin{align}\label{eq:Uniform_Expect}
    |\mathbb{E}[e^{-i \gamma_{pq} \tau}]| = \left|\sin\left(\frac{\gamma_{pq} \ell}{2}\right) \cdot \left(\frac{\gamma_{pq} \ell}{2}\right)^{-1}\right| = \left|\mbox{sinc}\left(\frac{\gamma_{pq} \ell}{2}\right) \right|
\end{align}
As $\ell \to \infty$, $|\mathbb{E}[e^{-i \gamma_{pq} \tau}]| \to 0$ \eqref{eq:Uniform_Expect} based on the behavior of the $\mbox{sinc}(\cdot)$ function. 
Then, $R_{p,\infty}$, as a finite linear combination of functions that converge to zero as $\ell \to \infty$ also converges to zero, the conclusions of Theorem \ref{thm:Gen_Van_Mat_Det_Gerschgorin} hold.
\smallskip

\subsubsection{Exponential Distribution}
Suppose now that the delays $\{\tau_k\}_{k = 1}^\ell$ are i.i.d. exponential random variables, $\tau \sim \mbox{Exp}(\lambda)$, where the PDF is $f(\tau;\lambda) = \lambda e^{-\lambda \tau}$, valid for $\tau \ge 0$ and with rate parameter $\lambda > 0$. 
To fit $1 - \varepsilon$ of the distribution in the interval $[0,\ell]$, pick $\lambda = \frac{-\ln(\varepsilon)}{\ell}$.
Then,
\begin{align}
    \mathbb{E}[e^{-i \gamma_{pq} \tau}] &= \int_{0}^\infty e^{-i \gamma_{pq} \tau} \lambda e^{-\lambda \tau} \dd \tau
    = \frac{\lambda}{\lambda + i \gamma_{pq}}
\end{align}
Taking the modulus of the above, we have:
\begin{align}
    |\mathbb{E}[e^{-i \gamma_{pq} \tau}]| 
    = |\lambda|/\sqrt{\lambda^2 + \gamma_{pq}^2}
\end{align}
Substituting $\lambda = \frac{-\ln(\varepsilon)}{\ell}$ and taking the limit as $\ell \to \infty$,
\begin{align}
    \lim_{\ell \to \infty} |\lambda|/\sqrt{\lambda^2 + \gamma_{pq}^2} = \lim_{\ell \to \infty} \frac{-\ln(\varepsilon)}{\sqrt{\ln(\varepsilon)^2 + \ell^2 \gamma_{pq}^2}} = 0
\end{align}
and thus $R_{p,\infty} \to 0$ and the conclusions of Theorem $\ref{thm:Gen_Van_Mat_Det_Gerschgorin}$ apply.
\smallskip

\subsubsection{Normal Distribution}
If the delays $\{\tau_k\}_{k = 1}^{\ell}$ are i.i.d. normally-distributed random variables, $\tau \sim \cN(\mu, \sigma^2)$, we can pick $\mu = \ell/2$ and $\sigma^2 = \alpha \ell^2$ for some $\alpha > 0$ such that a high percentage of the distribution is contained within $[0,\ell]$. 
Of course, unlike a truncated normal distribution, this would allow, albeit with low probability, negative sampling times, which would not be realistic.
However, for the purpose of analysis, it is sufficiently reasonable.
Then,
\begin{align}
    \mathbb{E}[e^{-i \gamma_{pq} \tau}] &= \int_{-\infty}^\infty e^{-i \gamma_{pq} \tau} \frac{1}{\sqrt{2 \pi \sigma^2}} \exp\left(- \frac{(\tau-\mu)^2}{2 \sigma^2}\right) \dd \tau \nonumber \\
    &= \frac{1}{\sqrt{2 \pi \sigma^2}} \int_{-\infty}^\infty e^{-i \gamma_{pq} \tau - \frac{(\tau- \mu)^2}{2 \sigma^2}} \dd \tau \nonumber \\
    &= \frac{1}{\sqrt{2 \pi \sigma^2}} \exp(-i \mu \gamma_{pq} - \frac{1}{2} \sigma^2 \gamma_{pq}^2) \label{eq:Normal_Expect}
\end{align}
Taking the modulus of \eqref{eq:Normal_Expect} cancels oscillations, and by \eqref{eq:R_p_inf}:
\begin{align}
    R_{p,\infty} &= m \sum_{q = 1, p \neq q}^m \frac{1}{\sqrt{2 \pi \sigma^2}} \exp(- \frac{1}{2} \sigma^2 \gamma_{pq}^2)
\end{align}
When $\ell \to \infty$, $\sigma^2 \to \infty$ too and the distribution escapes to infinity so that the CDF for any finite value is zero, in which case $R_{p,\infty} \to 0$ for all $p \in \{1,...,m\}$, thereby leading to the conclusions of Theorem \ref{thm:Gen_Van_Mat_Det_Gerschgorin}.
\smallskip

\begin{figure}[!t]
    \centering
    \includegraphics[width=3.5in]{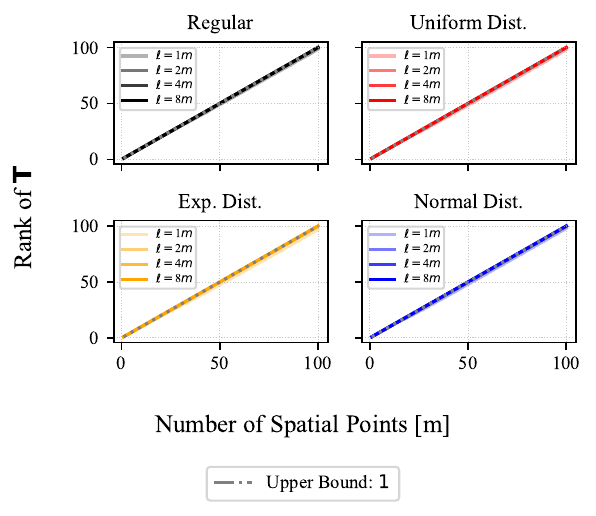}
    \caption{The rank of generalized Vandermonde matrix \eqref{eq:Generalized_Vandermonde_Matrix} in relation to number of frequencies $m$. 
    As $\ell \to \infty$, $\bT$ is full rank. 
    For small $\ell$, the normal and exponential distributions exhibit a slight deviation from full rank, likely a numerical artifact from the severely ill-posed nature of Vandermonde matrices.}
    \label{fig:Gen_Vand_Mat_Rank}
\end{figure}

\begin{figure}[!t]
    \centering
    \includegraphics[width=3.5in]{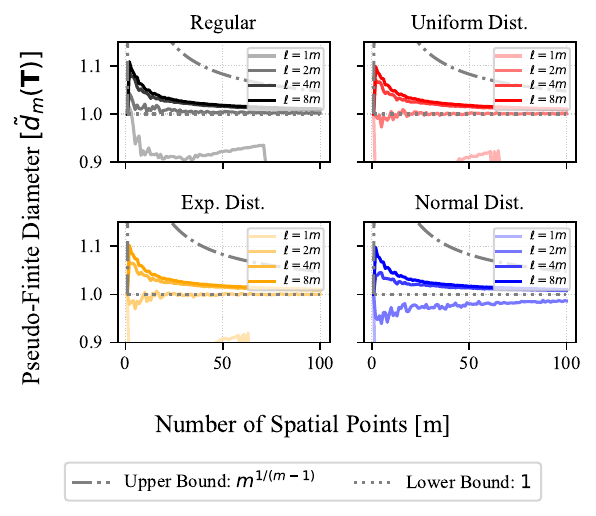}
    \caption{The scaled determinant \eqref{eq:Pseudo_Finite_Diameter} of the generalized Vandermonde matrix \eqref{eq:Generalized_Vandermonde_Matrix} is shown for different ratios of the number of lags to the number of sample points.
    In each case when $\ell = m$, whether the sampling is deterministic or random, the determinant is less than $1$, and its value drops suddenly around $m = 60$, most likely due to a numerical artifact stemming from the ill-posed nature of Vandermonde matrices.
    Around $\ell = 2m$, the determinant tends to oscillate around and converge towards $1$ as $m$ increases, especially for the uniformly and exponentially-distributed sampling schemes. 
    When $\ell \gg m$, all the distributions have a scaled determinant greater than $1$ that, in the limit as $m \to \infty$, converge to $1$.}
    \label{fig:Gen_Vand_Mat_Det}
\end{figure}

\begin{figure}[!t]
    \centering
    \includegraphics[width=3.5in]{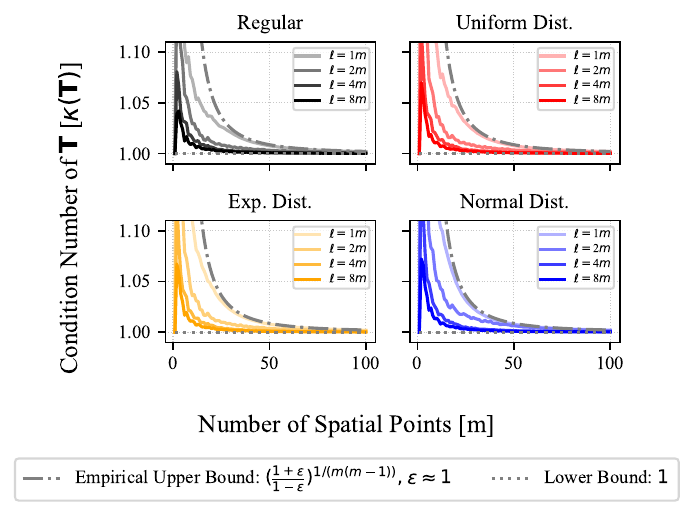}
    \caption{The scaled condition number of the generalized Vandermonde matrix \eqref{eq:Generalized_Vandermonde_Matrix}. 
    Increasing $\ell/m$ correlates with the scaled condition number for generalized Vandermonde matrices converging to $1$ faster in the deterministic case, and for random cases using uniform, normal, and exponential distributions.
    The upper bound for the condition number \eqref{eq:Gen_Vand_Mat_Cond_Bound} is empirically determined by finding a suitable $\varepsilon$, which in this case is $\varepsilon = 1-10^{-9}$. 
    Decreasing $\varepsilon$ by orders of magnitude generates tighter bounds for higher $\ell/m$ ratios.}
    \label{fig:Gen_Vand_Mat_Cond}
\end{figure}

\subsubsection{Numerical Simulations}
To empirically test the propositions in Conjecture \ref{conj:Gen_Vand_Rank}, we generated random conjugate pairs of frequencies $\{\pm \theta_j\}_{j = 1}^d$ for $d \in \{1,...,50\}$ drawn uniformly from $[0,\pi]$. 
For each choice of $d$, we performed $50$ iterations to generate a representative average and variance for the results.
To find the rank, scaled determinant, and condition number of $\sqrt{m/\ell} \bT$, we computed the singular values of the matrix using the SVD.
The rank is the number of nonzero singular values; the condition number is the ratio of the largest to the smallest singular value; and the scaled determinant is the product of the singular values normalized by $1/\binom{m}{2}$.

Figure \ref{fig:Gen_Vand_Mat_Rank} illustrates the rank of $\bT$; Figure \ref{fig:Gen_Vand_Mat_Det} depicts the normalized determinant; and Figure \ref{fig:Gen_Vand_Mat_Cond} shows the normalized condition number for different numbers of points $m$ in the regime where $\ell \to \infty$.
For each property, the plot compares the result for the generalized Vandermonde matrix for the 3 chosen PDFs (uniform, exponential, and normal) with that of the normal Vandermonde matrix, which is the baseline.
Empirically, $\bT$ is always full rank, and the normalized condition number and rank of the generalized Vandermonde matrices converge to $1$ as $m \to \infty$.
The normalized determinant upper bound of $m^{1/(m-1)}$ is an overestimate for when $\ell \ll +\infty$.
Thus, empirical results seem to suggest that generalized Vandermonde matrices tend to behave similarly to regular Vandermonde matrices with consecutive integer powers.

\subsection{Non-Uniformly-Sampled Takens' Embedding Theorem}

If Conjecture \ref{conj:Gen_Vand_Rank} is true, we can establish two theorems about delay-coordinate maps consisting of non-uniformly-sampled time delays.
Just as Theorem \ref{thm:Stable_Takens_Embedding_1} establishes the \textit{existence} of a stable embedding for a delay-coordinate map with evenly-spaced lags for linear dynamical systems, so too can we establish the \textit{existence} of a stable embedding for delay-coordinate maps of non-uniformly-spaced lags in the following theorem.
\medskip

\begin{theorem} \label{thm:Nonuniform_Stable_Takens_Embedding_1}
    (Linear Takens' Embedding with Non-Uniform Sampling) Consider a linear dynamical system of class $\cA(d)$ in $\R^n$ in steady-state. 
    Let $\{\tau_k\}_{k = 1}^\ell > 0$ be distinct times, $\bh \in \R^n$ be the observation function, and $\bPsi$ the delay-coordinate map as in \eqref{eq:Nonuniform_Delay_Coordinate_Map}. 
    Suppose that $\ell \ge 2d$, the $\cA_{\bPhi}$-eigenvalues $\{e^{\pm i \theta_j \tau_k}\}$ are distinct and purely imaginary, and $\bv_j^H \bh \neq 0$ for all $j \in \{1,...,d\}$. 
    Then, for all distinct pairs of points $\bx_1, \bx_2 \in \cM$, with probability one $\bPsi$ satisfies \eqref{eq:Stable_Embedding} for some constants $C$ and $\delta < 1$.
\end{theorem}
\bigskip

\begin{IEEEproof}
    If Conjecture \ref{conj:Gen_Vand_Rank} is true, then the result follows from a nearly identical approach to the proof of Theorem III.1 of \cite{Yap_2011}. 
    In particular, provided $\rank(\bT) = \min\{\ell, 2d\}$, then for $\ell \ge 2d$, $\rank(\bT) = 2d$, and, provided $\rank(\bH) = 2d$, then $\rank(\bG) = 2d$.
    This is because $\rank(\bG) = \rank(\bH) - \dim(\cN(\bT)\cap\cR(\bH))$, where $\cN(\bT)$ denotes the null-space of $\bT$ and $\cR(\bH)$ is the range of $\bH$;
    when $\rank(\bH) = 2d$, $\cR(\bH)$ is all of $\R^d$, and when $\rank(\bT) = 2d$, $\cN(\bT)$ is empty, leading to $\rank(\bG) = 2d$.
    Thus, when $\ell \ge 2d$, the entire attractor $\cM$ is embedded into the delay-coordinate space by $\bPsi$.
\end{IEEEproof}
\bigskip

Likewise, Theorem \ref{thm:Stable_Takens_Embedding_2}, which provides explicit enveloping bounds and long-term asymptotic bounds for the stability of linear delay-coordinate maps with uniformly-spaced time-delays \eqref{eq:Uniform_Delay_Coordinate_Map}, can be extended partially to the setting of delay-coordinate maps with non-uniformly-spaced delays \eqref{eq:Nonuniform_Delay_Coordinate_Map}.
Due to the unevenly-sampled delay structure of generalized Vandermonde matrices $\bT$, it is unknown whether explicit enveloping upper and lower bounds on the embedding quality can be derived.
What is possible is to capture the asymptotic bounds for stability as $\ell \to \infty$ to show that $\bPsi$ can approach isometry.
\medskip

\begin{theorem} \label{thm:Nonuniform_Stable_Takens_Embedding_2}
    (Stable Linear Takens' Embedding with Non-uniform Sampling; Theorem III.2 \cite{Yap_2011}) Consider a linear system of class $\cA(d)$ in $\R^n$ that is in steady-state. 
    Let $\{\tau_k\}_{k = 1}^{\ell} > 0$ be the sampling times, $\bh \in \R^n$ be the observation function such that $\|\bh\|_2^2 = \frac{2d}{\ell}$, and $\bPsi$ the delay-coordinate map with $\ell$ delays defined as in \eqref{eq:Nonuniform_Delay_Coordinate_Map}. 
    Suppose that $\ell \gg 1$, the $\cA_{\bPhi}$-eigenvalues $\{e^{\pm i \theta_j \tau_k}\}$ are distinct and purely imaginary, and $\bv_j^H \bh \neq 0$ for all $j \in \{1,...,d\}$. 
    Then, for all distinct pairs of points $\bx_1, \bx_2 \in \cM$, with probability one $\bPsi$ satisfies \eqref{eq:Stable_Embedding} in the limit as $\ell \to \infty$ with constants $C$ and $\delta \to \delta_0$:
    \begin{align}
        C:= d \left(\frac{\kappa_1^2}{A_2} + \frac{\kappa_2^2}{A_1}\right) \qquad \delta_0 := \frac{A_2 \kappa_2^2 - A_1 \kappa_1^2}{A_2 \kappa_2^2 + A_1 \kappa_1^2}
    \end{align}
\end{theorem}
\medskip

\begin{IEEEproof}
    The proof follows from the same approach as Theorem III.2 in \cite{Yap_2011}, with the exception that we neglect the lag-dependent term $\delta_1(\ell)$ because we are considering only the asymptotic behavior of the embedding.
\end{IEEEproof}
\bigskip

The difficulty of establishing a $\delta_1(\ell)$-term in Theorem \ref{thm:Nonuniform_Stable_Takens_Embedding_2} equivalent to that of Theorem \ref{thm:Stable_Takens_Embedding_2} chiefly lies in the non-geometric nature of the non-uniform time series data.
Whereas uniformly-sampled delay coordinate maps generate a geometric sum of exponential terms when applying Gerschgorin's theorem to $\bS = (m/\ell)\bT^H \bT$ as seen in Lemma A.1 in \cite{Yap_2011}, the geometric structure is lost when using non-uniformly-sampled data, and the expression \eqref{eq:Gerschgorin_Rad} does not readily simplify.
\medskip

\begin{remark}
    To introduce an additional degree of non-uniform sampling in the linear setting beyond the lag intervals between each delay, we can also allow for the observation function to be time-dependent, similar to \cite{Huke_2007}.
    Consider a time-dependent observation function $\bh(t) \in \R^n$ for which $\bv_j^H \bh(t) \neq 0$ for all $j \in \{1,...,d\}$ and $t \in \R$.
    The resulting non-uniformly-sampled delay-coordinate map would be:
    \begin{align} \label{eq:Nonuniform_Delay_Coordinate_Map_2}
        \bPsi_{\{\bh(t), \bPhi, \{\tau_k\}_{k = 1}^\ell\}}(\bx(t))
        = \begin{bmatrix}
            \bh(t)^T \bPhi^{-\tau_1}\bx(t) \\ \bh(t)^T \bPhi^{-\tau_2} \bx(t) \\ \vdots \\ \bh(t)^T \bPhi^{-\tau_\ell} \bx(t)
        \end{bmatrix}
    \end{align}
    Suppose $\bh(t)$ is still normalized by requiring $\|\bh(t)\|_2^2 = \frac{2d}{\ell}$. 
    The time-dependence of the linear observation vector implies the measurement apparatus may be rotating or translating in time as it samples the dynamics, and the normalization ensures there is no stretching or shrinking. 
    This may well be applicable when using mobile sensors that move relative to the system being observed. 
    In the linear setting, the dynamics described by the embedding matrix $\bG(t)$ can still be decoupled into a Vandermonde matrix $\bT$ exactly as in \eqref{eq:Generalized_Vandermonde_Matrix} and a diagonal matrix $\bH(t)$ each element of which--$\bv_j^H \bh(t)$ and $\bv_j^T \bh(t)$--is now time-dependent.
    Consequently, the stable embedding bounds for the delay-coordinate map with mobile sensors will also be time-dependent.
    To make them time independent, the bounds must account for the behavior of $\bh(t)$ over time. Re-defining $\kappa_1$ and $\kappa_2$ as in \eqref{eq:kappa_vals} to be:
    \begin{align*}
        \kappa_1 = \min_{t \in \R} \min_{j \in \{1,...,d\}} \frac{|\bv_j^H \bh(t)|}{\|\bh(t)\|_2} \quad
        \kappa_2 = \max_{t \in \R} \max_{j \in \{1,...,d\}} \frac{|\bv_j^H \bh(t)|}{\|\bh(t)\|_2}
    \end{align*}
    would allow the conclusions of Theorem \ref{thm:Nonuniform_Stable_Takens_Embedding_1} and \ref{thm:Nonuniform_Stable_Takens_Embedding_2} to apply with these new definitions of $\kappa_1$, $\kappa_2$, and $\bh(t)$.
\end{remark}

\subsection{Numerical Simulations}

In this section, we implement simulated numerical experiments to assess whether delay-coordinate maps for linear systems with unevenly-spaced delays actually approach the asymptotic embedding quality $\delta_0$.
The approach to the experiments is mostly consistent with \cite{Yap_2011}.
We generate a dynamical system of ambient dimension $n = 50$ with attractor frequencies $\{\pm\theta_j\}_{j = 1}^d$, with $d = 4$ for our example system and where each $\theta_j$ is drawn from a random uniform distribution over $(0,2\pi)$.
Considering that a sufficient condition on the sampling frequency to guarantee that $\bPsi$ will embed the dynamics of the original system is to choose $\tau^* < \frac{\pi}{\theta_{\max}}$, we selected the sampling interval to be $\tau = 0.8 \tau^*$ for the deterministic, uniformly-sampled delay coordinate map, which serves as a baseline comparison for the non-uniformly-spaced case. (For the numerical simulations, $\tau = 0.9864$.)

For the delay-coordinate maps with unevenly-spaced delays, we generated the delays by initializing a length $\ell_{\max} = 500$ delay vector with entries drawn randomly from the following distributions: uniform, exponential, or normal.
For each $\ell \in \{1,...,\ell_{\max}\}$, a delay-coordinate map was generated using the first $\ell$ lags of the $\ell_{\max}$ total entries of the delay vector.
This ensured that each lag test was consistent, in that it preserved the previous time lags when generating a delay-coordinate embedding with additional delay coordinates.
\smallskip

\subsubsection{Uniform}
To ensure the distribution was contained in the interval $[0,\ell]$ for various choices of $\ell$, we drew a length $\ell_{\max}$ vector uniformly from $[0,1]$ and then scaled the first $\ell$ entries used for the particular iteration by $\ell$.
\smallskip

\subsubsection{Exponential}
Instead of using a typical exponential distribution, the domain of which is $[0,\infty)$, we used a truncated exponential distribution (using SciPy's \texttt{truncexpon} function) contained in the interval $[0,1]$ as a means to initialize a length-$\ell_{\max}$ vector the first $\ell$-elements we would scale by $\ell$ for the particular test iterate.
\smallskip

\subsubsection{Normal}
As opposed to using a typical normal distribution, we used a truncated normal distribution (using SciPy's \texttt{truncnorm} function) with $\mu = 0.5$ and $\sigma^2 = 0.2$ and with the lower and upper bounds being $0$ and $1$, respectively.
\smallskip

To empirically measure how stable a particular embedding is, we consider the embedding conditioning: 
\begin{align}\label{eq:Empirical_Q}
    Q(\bx_1, \bx_2) = \frac{\|\bPsi(\bx_1) - \bPsi(\bx_2)\|_2^2}{\|\bx_1 - \bx_2\|_2^2}
\end{align}
To empirically compute $Q(\bx_1, \bx_2)$ for a particular $\ell$, we perform $5,000$ trials comparing the difference of pairs of states $\bx_1$ and $\bx_2$ under the delay-coordinate map.
For each trial, we generate random pairs of points $\bx_1$ and $\bx_2$ from the same solution trajectory on the attractor.
To do so, we draw pairs of random sample times $t_{\bx_1}$ and $t_{\bx_2}$ uniformly from $[0,10000]$, and initialize the trajectory with the initial condition $\bx_0 = \bV \balpha_0$, where the coefficient vector $\balpha_0 = \begin{bmatrix}
    1 & 1 & \cdots & 1 & 1
\end{bmatrix} \in \C^{2d}$ to ensure that the chosen solution captures information on all $2d$ dimensions of the attractor.
By Definition \ref{def:Attractor}, the pairs of states are: $\bx_1(t) = \bV e^{\bLambda t_{\bx_1}} \balpha_0$ and $\bx_2(t) = \bV e^{\bLambda t_{\bx_2}} \balpha_0$.
For each $\ell$ ranging from $1$ to $500$, we identify the maximal and minimal embedding quality $\max\{Q\}$ and $\min\{Q\}$ to describe how the range of the embedding quality changes with the number of lags used in the delay-coordinate map.
We also compute the variance as another metric of the typical embedding quality.
\smallskip

\subsubsection{Case 1 (Ideal Dynamics and Ideal Measurement)}
Theorem \ref{thm:Nonuniform_Stable_Takens_Embedding_2} suggests that the delay-coordinate map is an isometry when the conditioning parameter $\delta_0 = 0$, which occurs when $A_1 = A_2$ and $\kappa_1 = \kappa_2$.
When $A_1 = A_2$, it implies that the underlying attractor $\cM$ is ``uniformly-stretched" in each dimension, and $\kappa_1 = \kappa_2$ implies that $\bh$ observes each of the dimensions of $\cM$ with the same magnitude.
To numerically simulate this idealized scenario as a way to confirm Theorem \ref{thm:Nonuniform_Stable_Takens_Embedding_2}, we ensure $A_1 = A_2$ by generating the $\cA_{\bPhi}$-eigenvectors with $\bv_j = \frac{\sqrt{2}}{2} (\be_{2j-1} + i \be_{2j})$, where $\{\be_j\}_{j = 1}^d$ are the canonical basis vectors in $\R^n$, so that $\bv_j$ and $\bv_j^H$ form a conjugate pair.

For an observation function that ensures $\kappa_1 = \kappa_2$, define $\bc = \bV \boldsymbol{1}$, where $\boldsymbol{1} = \begin{bmatrix}
    1 & 1 & \cdots & 1 & 1
\end{bmatrix}^T$, so that the vector observes the same magnitude of each dimension of the attractor; then, set $\bh = \sqrt{\frac{2d}{\ell}} \frac{\bc}{\|\bc\|_2}$ for the specific choice of $\ell$.

\begin{table}[!t]
    \caption{Case 1\label{tab:Case_1}}
    \centering
    \begin{threeparttable}
    
    \begin{tabular}{|c||c|c|c|c|}
        \hline
        $\cA$-Dimension Index: $j$ & 1 & 2 & 3 & 4 \\
        \hline
        $\theta_j$ & 2.5480 & 1.5158 & 1.6468 & 0.1489 \\
        \hline
        $|\bv_j^H \bh|^2/\|\bh\|_2^2$ & 1 & 1 & 1 & 1 \\
        \hline
        $\lambda_j(\bV^H \bV)$ & 1 & 1 & 1 & 1 \\
        \hline
    \end{tabular}

    \begin{tablenotes}
        \footnotesize
        \item $A_1 = 1$, $A_2 = 1$, $\kappa_1 = 1$, $\kappa_2 = 1$, $C = 1$, $\delta_0 = 0$, and $\nu = 15.4879$.
    \end{tablenotes}

    \end{threeparttable}
\end{table}

\begin{figure}[!t]
    \centering
    \includegraphics[width=3.5in]{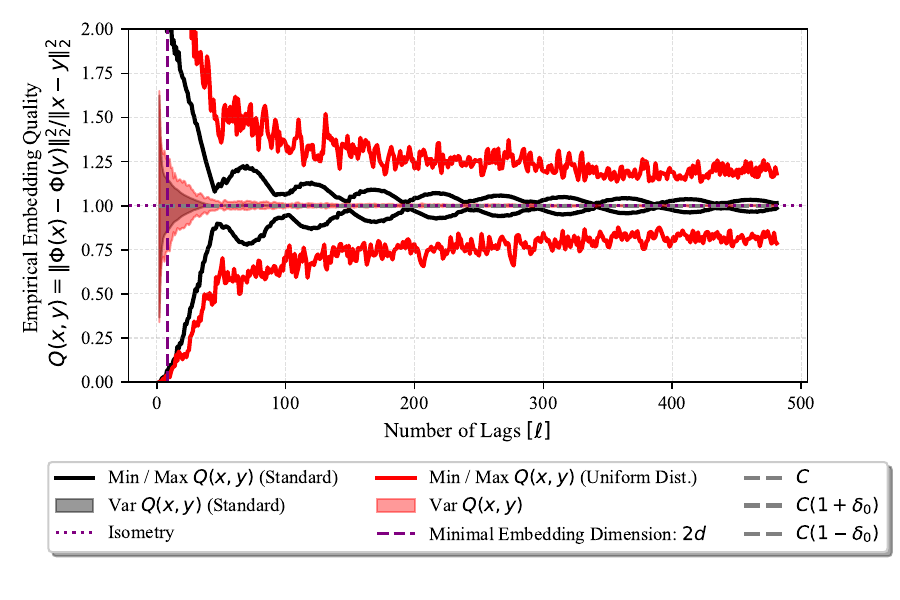}
    \caption{Numerical simulation of the conditioning of delay-coordinate maps with delays drawn uniformly from $[0,\ell]$ in the case when $A_1 = A_2$ and $\kappa_1 = \kappa_2$ (and thus $\delta_0 = 0$), compared to the deterministic delay-coordinate map with evenly-spaced delays and the theoretical bounds from Theorem \ref{thm:Nonuniform_Stable_Takens_Embedding_2}.
    In the asymptotic regime, the random and deterministic delay-coordinate maps approach isometry.
    While $\max\{Q\}$ and $\min\{Q\}$ can be quite large, by comparison $\mbox{var}\{Q\}$ is quite small.}
    \label{fig:Takens_Uniform_Case_1}
\end{figure}

\begin{figure}[!t]
    \centering
    \includegraphics[width=3.5in]{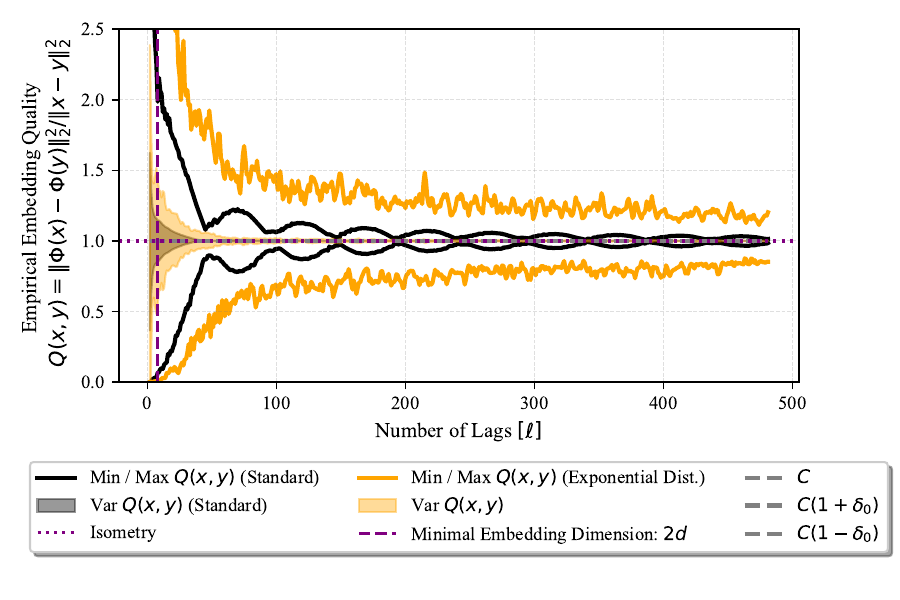}
    \caption{Numerical simulation for Case $1$, where $A_1 = A_2$ and $\kappa_1 = \kappa_2$, this time for delay-coordinate maps with delays distributed according to an exponential r.v. between $[0,\ell]$. 
    The maximal and minimal embedding quality begins to converge to $C$ as $\ell$ increases, potentially allowing for isometric time-delay embeddings.
    The variance of $Q$ is also quite low.}
    \label{fig:Takens_Exp_Case_1}
\end{figure}

\begin{figure}[!t]
    \centering
    \includegraphics[width=3.5in]{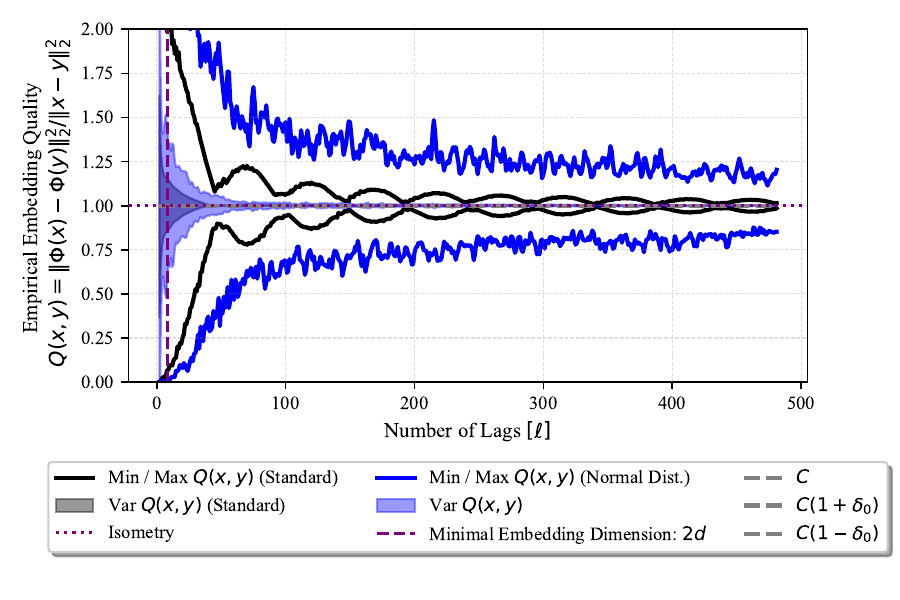}
    \caption{Numerical simulation for Case $1$, where $A_1 = A_2$ and $\kappa_1 = \kappa_2$, now where the delay-coordinate maps have normally-distributed delays on the truncated interval $[0,\ell]$.
    Again, the variance of $Q$ is low, and $Q$ converges to $C$ as $\ell$ increases.}
    \label{fig:Takens_Normal_Case_1}
\end{figure}

The simulation results are depicted in Figures \ref{fig:Takens_Uniform_Case_1}, \ref{fig:Takens_Exp_Case_1}, \ref{fig:Takens_Normal_Case_1}, and the parameters are listed in Table \ref{tab:Case_1}.
Each figure presents the embedding quality of a non-uniformly-sampled delay-coordinate map \eqref{eq:Nonuniform_Delay_Coordinate_Map}--uniform, exponential, or normal--in the form of statistics $\max\{Q(\bx_1,\bx_2)\}$, $\min\{Q(\bx_1, \bx_2)\}$, and $\mbox{var}\{Q(\bx_1,\bx_2)\}$ with $Q(\bx_1,\bx_2)$ as in \eqref{eq:Empirical_Q}, and compares the results against the uniformly-sampled delay-coordinate map \eqref{eq:Uniform_Delay_Coordinate_Map} as a baseline, as well as against the theoretical bounds for asymptotic embedding constants $C$, $C(1 + \delta_0),$ and $C(1-\delta_0)$.

Importantly, $\max\{Q\}$ and $\min\{Q\}$ do begin to converge to $C$ for each of the non-uniformly-sampled delay-coordinate maps, illustrating that increasing the number of time-delayed coordinates $\ell$ in the embedding can improve its quality, which seems consistent with Theorem \ref{thm:Nonuniform_Stable_Takens_Embedding_2}.
However, the non-uniformly-sampled maps consistently perform more poorly than the uniformly-spaced delay-coordinate map, at least in this set-up, and, without an enveloping bound generated by a $\delta_1(\ell)$, providing a rate of convergence is impossible.
\smallskip

\subsubsection{Case 2 (Ideal Dynamics and Non-Ideal Measurement)}
Next, we simulate the case where the attractor is still ``ideal" in the sense that $A_1 = A_2$, but now where the observation function $\bh$ does not observe each dimension of the attractor in the same way, meaning $\kappa_1 \neq \kappa_2$.
We generate the $\cA_{\bPhi}$-eigenvectors in the same manner as in Case $1$, but now for the observation function, we generate $\bc \in \R^n$ by letting $\bc = \sum_{j = 1}^d ((1 + w_{2j-1}) \Re(\bv_j) + (1 + w_{2j}) \Im(\bv_j))$, with the $\{w_j\}_{j = 1}^{2d}$ being i.i.d. Gaussian random variables of mean $0$ and variance $0.1$.
For each value $\ell$, we generate $\bh = \sqrt{\frac{2d}{\ell}} \frac{\bc}{\|\bc\|_2}$.
Since each $\{w_j\}_{j=1}^{2d}$ is of low variance, $\frac{|\bv_j^H \bh|}{\|\bh\|_2^2}$ is close to $1$, in which case $\kappa_1$ and $\kappa_2$ are close to but typically not equal to $1$.
With $A_1 = A_2$ and $\kappa_1$ and $\kappa_2$ close to $1$, $\delta_0$ is small.

\begin{table}[!t]
    \caption{Case 2\label{tab:Case_2}}
    \centering

    \begin{threeparttable}

    \begin{tabular}{|c||c|c|c|c|}
        \hline
        $\cM$-Dimension Index: $j$ & 1 & 2 & 3 & 4 \\
        \hline
        $\theta_j$ & 2.5480 & 1.5158 & 1.6468 & 0.1489 \\
        \hline
        $|\bv_j^H \bh|^2/\|\bh\|_2^2$ & 1.0409 & 1.0547 & 0.9339 & 0.9654\\
        \hline
        $\lambda_j(\bV^H \bV)$ & 1 & 1 & 1 & 1 \\
        \hline
    \end{tabular}

    \begin{tablenotes}
        \footnotesize
        \item $A_1 = 1$, $A_2 = 1$, $\kappa_1 = 0.9339$, $\kappa_2 = 1.0547$, $C = 0.7886$, $\delta_0 = 0.1210$, and $\nu = 15.4879$.
    \end{tablenotes}
    
    \end{threeparttable}
\end{table}

\begin{figure}[!t]
    \centering
    \includegraphics[width=3.5in]{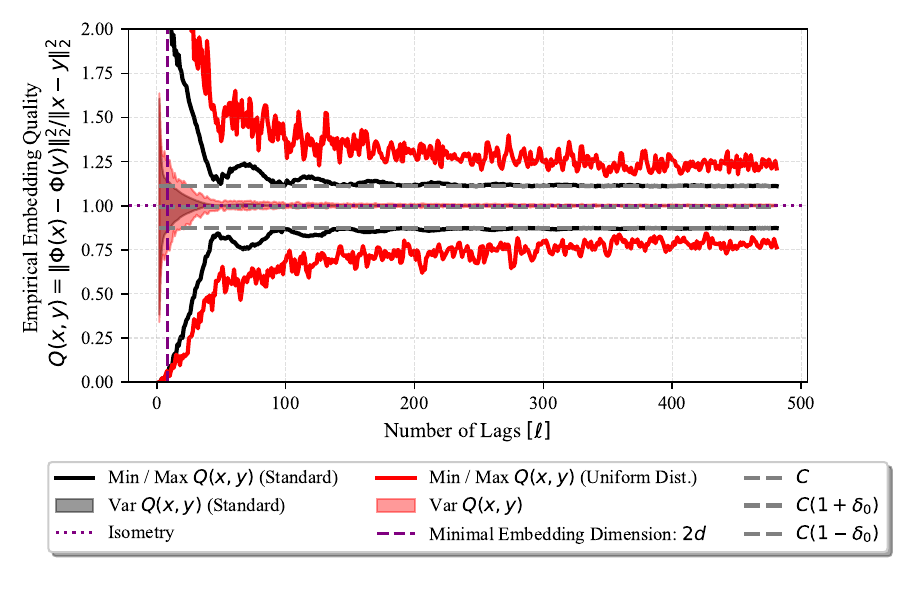}
    \caption{Numerical simulation of the conditioning of delay-coordinate maps with delays drawn uniformly from $[0,\ell]$ in the case when $A_1 = A_2$ and $\kappa_1 \neq \kappa_2$ (and thus $\delta_0 > 0$), with the deterministic delay-coordinate map with evenly-spaced delays as a reference and the bounds from Theorem \ref{thm:Nonuniform_Stable_Takens_Embedding_2}.
    As $\ell$ grows large, the random and deterministic maps approach isometry, and, although $\max\{Q\}$ and $\min\{Q\}$ can be large, $\mbox{var}\{Q\}$ is still relatively small.}
    \label{fig:Takens_Uniform_Case_2}
\end{figure}

\begin{figure}[!t]
    \centering
    \includegraphics[width=3.5in]{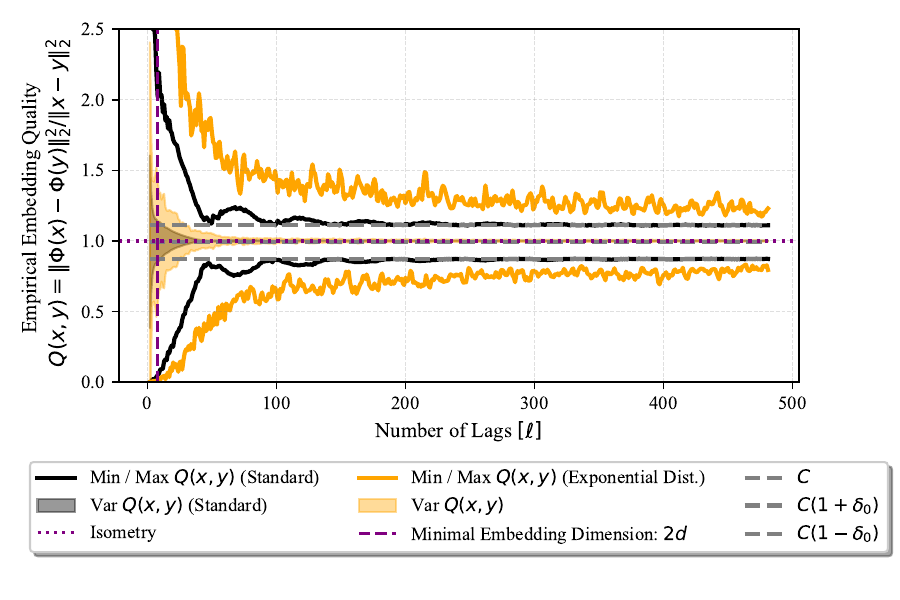}
    \caption{The stability of non-uniform linear Takens' embedding in the case where $A_1 = A_2$ but where $\kappa_1 \neq \kappa_2$ where the time delays adhere to an exponential distribution. 
    Again, the non-deterministically-generated map performs poorer than its deterministic, evenly-spaced counterpart, but as $\ell \to \infty$ they do begin to converge to the same asymptotic bounds. 
    Again $\var\{Q\}$ is comparably small to $\max\{Q\}$ and $\min\{Q\}$, so most of the embedded attractor is isometric to the original attractor.}
    \label{fig:Takens_Exp_Case_2}
\end{figure}

\begin{figure}[!t]
    \centering
    \includegraphics[width=3.5in]{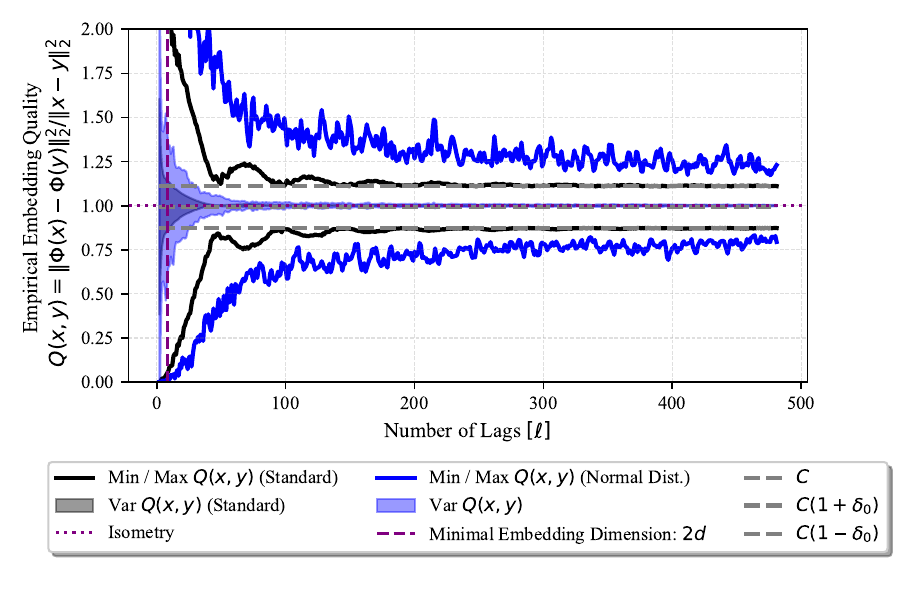}
    \caption{Stability of non-uniform Linear Takens' Embedding where time lags adhere to a normal distribution. Low variance}
    \label{fig:Takens_Normal_Case_2}
\end{figure}

Figures \ref{fig:Takens_Uniform_Case_2}, \ref{fig:Takens_Exp_Case_2}, \ref{fig:Takens_Normal_Case_2} depict the results, and Table \ref{tab:Case_2} contains the parameters.
Again, $\max\{Q\}$ and $\min\{Q\}$ do begin to converge to $C(1 \pm \delta_0)$ for each of the non-uniformly-sampled delay-coordinate maps, but because $\kappa_1 \neq \kappa_2$, there is a fundamental limit to the quality achievable to the time-delay embedding and it cannot achieve isometry.
This confirms that increasing the lags $\ell$ in the delay-coordinate map does improve the embedding quality, but up to a limit, an observation in agreement with Theorem \ref{thm:Nonuniform_Stable_Takens_Embedding_2}.
Again, the normal delay-coordinate map \eqref{eq:Nonuniform_Delay_Coordinate_Map} outperforms maps with non-uniform delays, and there is no theoretical enveloping bound for a $\delta_1(\ell)$ that could provide convergence guarantees.
\smallskip

\subsubsection{Case 3 (Non-Ideal Dynamics and Non-Ideal Measurement)}
Finally, we simulate the case when $\delta_0$ may be larger, which is when $A_1 \neq A_2$ and $\kappa_1 \neq \kappa_2$.
We use the same $\cA_{\bPhi}$-eigenvalues and the same $\bh$ as in Case $1$.
For the attractor dynamics, we let $\{\bv_j\} = \frac{(\ba_j + i \bb_j)}{\sqrt{\|\ba\|_2^2 + \|\bb\|_2^2}}$, where $\{\ba_j, \bb_j\} \in \R^n$ are randomly-constructed vectors the entries of which are normally-distributed with mean $0$ and variance $1$.

\begin{table}[!t]
    \caption{Case 3\label{tab:Case_3}}
    \centering

    \begin{threeparttable}

    \begin{tabular}{|c||c|c|c|c|}
        \hline
        $\cM$-Dimension Index: $j$ & 1 & 2 & 3 & 4 \\
        \hline
        $\theta_j$ & 2.5480 & 1.5158 & 1.6468 & 0.1489 \\
        \hline
        $|\bv_j^H \bh|^2/\|\bh\|_2^2$ & 0.6581 & 0.2853 & 0.1888 & 0.4264 \\
        \hline
        $\lambda_j(\bV^H \bV)$ & 2.0433 & 1.6724 & 1.5849 & 1.3529 \\
        \hline
        $\lambda_{j+4}(\bV^H \bV)$ & 0.9879 & 0.1735 & 0.0779 & 0.1072 \\
        \hline
    \end{tabular}

    \begin{tablenotes}
        \footnotesize
        \item $A_1 = 0.2790$, $A_2 = 1.4294$, $\kappa_1 = 0.1888$, $\kappa_2 = 0.6581$, $C = 0.7886$, $\delta_0 = 0.9684$, and $\nu = 15.4879$.
    \end{tablenotes}

\end{threeparttable}

\end{table}

\begin{figure}[!t]
    \centering
    \includegraphics[width=3.5in]{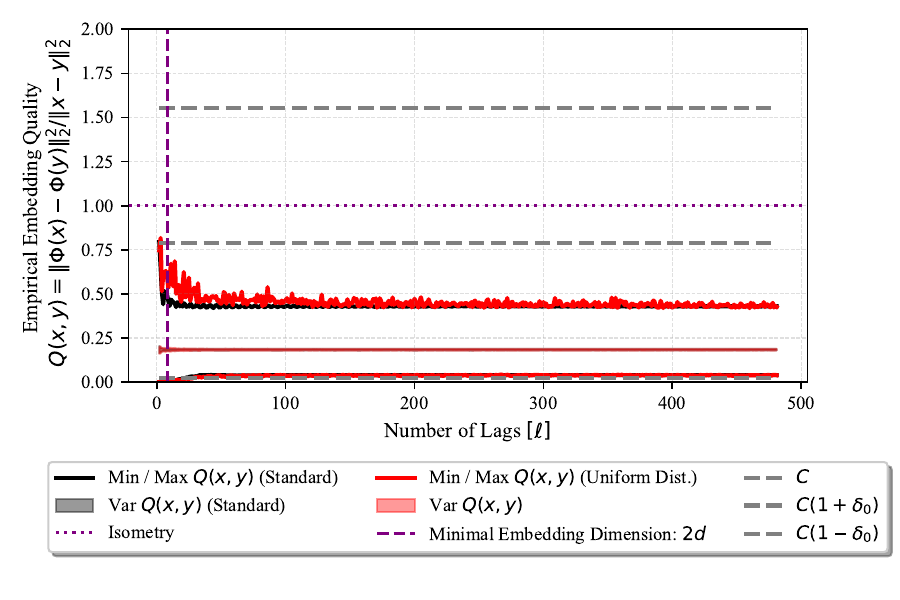}
    \caption{Simulation of the conditioning of delay-coordinate maps with delays drawn uniformly from $[0,\ell]$ when $A_1 \neq A_2$ and $\kappa_1 \neq \kappa_2$ ($\delta_0 > 0$).
    The deterministic map with evenly-spaced delays acts as a reference, and the theoretical bounds of Theorem \ref{thm:Nonuniform_Stable_Takens_Embedding_2} are plotted.
    The unevenly-sampled delay-coordinate map conditioning converges to the evenly-spaced case, with the normal mapping consistently outperforming it. 
    The empirical lower bound converges to $1 - \delta_0$, but the upper bound is not tight. 
    Without scaling by absorbing $C$ into $\bPhi$, the delay-coordinate map cannot be an isometric embedding.
    While the maximum and minimum quality range widely, the variance is small.}
    \label{fig:Takens_Uniform_Case_3}
\end{figure}

\begin{figure}[!t]
    \centering
    \includegraphics[width=3.5in]{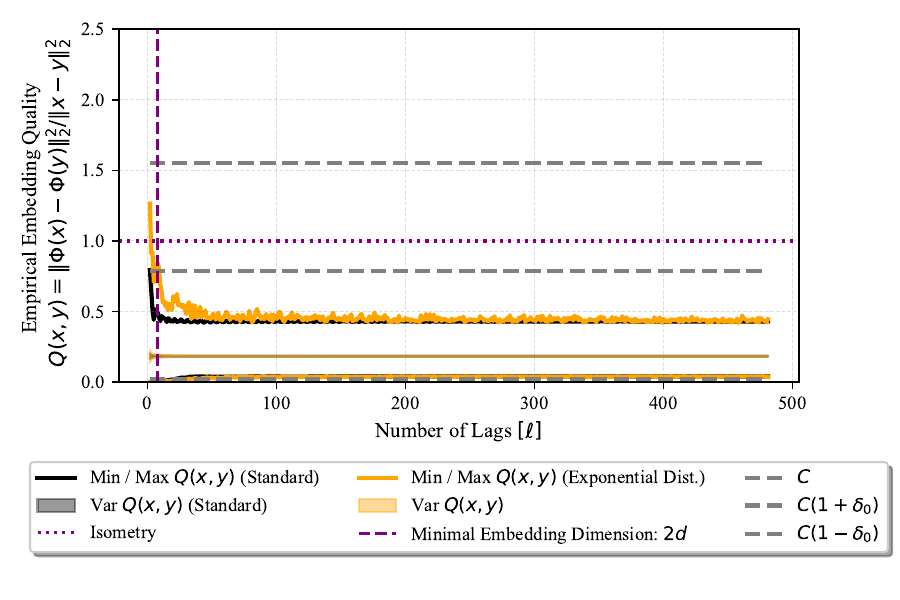}
    \caption{A numerical simulation similar to that of Figure \ref{fig:Takens_Uniform_Case_3}, with $A_1 \neq A_2$ and $\kappa_1 \neq \kappa_2$, except that the time lags used for the non-uniform delay-coordinate map are drawn from an exponential distribution.
    The embedding quality performs similarly to the map with deterministic, uniformly-spaced delays, as well as with the map with unevenly-spaced delays drawn uniformly.}
    \label{fig:Takens_Exp_Case_3}
\end{figure}

\begin{figure}[!t]
    \centering
    \includegraphics[width=3.5in]{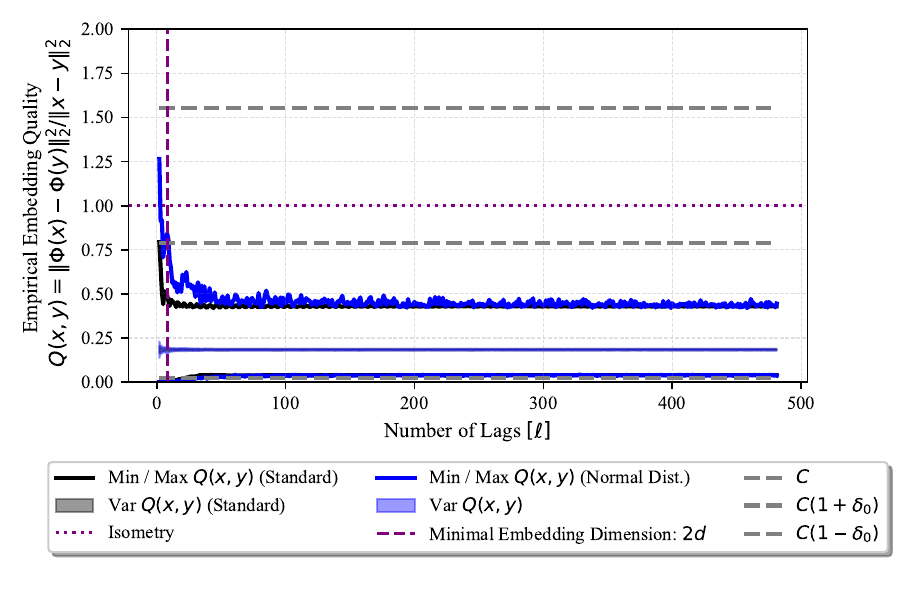}
    \caption{A numerical simulation like that of Figure \ref{fig:Takens_Uniform_Case_3}, but now with delays drawn from a normal distribution. 
    Again, because the eigenvectors of the underlying system are of differing magnitudes, and the difference of the magnitudes in the observed states varies greatly, the quality of the embedding is poor, yet it does not become singular.}
    \label{fig:Takens_Normal_Case_3}
\end{figure}

In this ``non-ideal" setting, the results for which are shown in Figures \ref{fig:Takens_Uniform_Case_3}, \ref{fig:Takens_Exp_Case_3}, \ref{fig:Takens_Normal_Case_3}, and the parameters for which are in Table \ref{tab:Case_3}, again the embedding is not isometric in the asymptotic case, although it could be improved by absorbing $C$ into $\bPhi$.
Again, the unevenly-spaced delay-coordinate map quality converges to that of the evenly-spaced delay-coordinate map, and the variance of the empirical quality is low.

Given that the deterministic delay-coordinate map with uniformly-spaced delays seems to empirically lower and upper bound the maximal and minimal values of the embedding quality $Q$ for the maps of non-uniformly-spaced delays, a few natural questions are: whether uniform-sampling tends to outperform non-uniform sampling schemes and in what kinds of applications and settings; or whether a better basis of forming delay-coordinate maps with non-uniform delays is needed to have a more suitable comparison to the uniform counterparts.
Another important question would be whether it is possible to establish estimates for an enveloping bound for the conditioning $\delta_1(\ell)$ so as to be able to offer guarantees of convergence and estimates of the rate of convergence.

\section{Conclusions}
We considered delay-coordinate maps constructed using unevenly-spaced time delays for linear systems evolving on their attractors and established that they can form stable embeddings.
To analyze the temporal component of the embeddings, we considered generalized Vandermonde matrices with distinct but unevenly-spaced powers \eqref{eq:Generalized_Vandermonde_Matrix}.
Just as having a distinct set of points and a distinct set of evenly-spaced integer powers are sufficient conditions for typical Vandermonde matrices \eqref{eq:Vandermonde_Matrix} to be full column rank when the number of lags exceeds the number of points, so too should having a distinct set of points and a distinct set of powers be sufficient for Vandermonde matrices with unevenly-spaced powers to be full column rank, but perhaps with the caveat that the statement would be one of high probability rather than a deterministic one.
As an initial step to confirm this, we proved theoretical results for when the number of time-delays $\ell$ is small and for the asymptotic regime as $\ell \to \infty$; we also provided empirical support for when such matrices are full rank.
Rigorously proving the conjecture that generalized Vandermonde matrices are full rank does remain an open problem, yet the results for the base cases, for the asymptotic case, and from the numerical simulations all support the proposition.

If generalized Fourier Vandermonde matrices with unevenly-spaced powers are full column rank with probability one if the system frequencies and time samples are distinct, then the theorems we proposed for the existence of stable embeddings using delay-coordinate maps of unevenly-spaced lags and for the explicit asymptotic bounds of the embedding stability hold true.
In the asymptotic regime where $\ell \to \infty$, the embedding quality of the delay-coordinate map generated for each of the chosen random time-sampling schemes converged to the same limiting asymptotic bound as that of the deterministic, uniformly-sampled delay-coordinate map.
This is to be expected in the linear setting, since, as shown in the deterministic case, the limiting behavior is dependent not on the number of time lags, but exclusively on the coordinates of the system being studied and the choice of observation vector.
Thus, just as increasing the number of lags used in a delay-coordinate map can mitigate the ill-conditioning of the embedding for a uniformly-sampled signal, so too in general does increasing the number of time lags in the non-uniform sampling setting improve the embedding quality.

However, the numerical simulations do suggest that standard, uniformly-spaced delay-coordinate maps tend to perform better than, and converge faster than, the non-uniform case. 
The theoretical results from Conjecture \ref{conj:Gen_Vand_Rank} detailing when a generalized Vandermonde matrix \eqref{eq:Generalized_Vandermonde_Matrix} has full column rank--at least with 1-3 distinct frequencies--agree with the numerical results: they show that more restrictive conditions are needed in addition to those for typical Fourier Vandermonde matrices \eqref{eq:Vandermonde_Matrix}.
Thus, that generalized Vandermonde matrices must satisfy more conditions than typical Vandermonde matrices indicates that uniformly-spaced delay-coordinate maps ought to perform better than their non-uniformly-spaced counterparts.
These insights lead to a number of further areas of investigation.
Deriving estimates for the rate of convergence would be a helpful metric for comparing uniformly-sampled and non-uniformly-sampled maps.
This also leads to an open question of whether uniform sampling is the guaranteed optimal sampling approach for linear systems that are not necessarily multi-scale.
Some future directions will be to examine stable linear embeddings using non-uniform time sampling for linear systems exhibiting frequencies across multiple timescales, and for linear systems observed using multiple observation functions as a means to embed multivariate time series.

From the viewpoint of applications, sensing strategies and signal processing techniques continue to be critically important for state-estimation and control in almost every domain of science and engineering.  
Leveraging the time-history of limited point sensors for reconstruction of full state spaces is clearly an important paradigm for practical deployment.
The non-uniform sampling of signals, especially in the context of emerging event-based sampling strategies, are also important for efficient data acquisition strategies.
The current work aims to establish rigorous estimates and bounds on when non-uniform sampling can guarantee faithful state-space estimation, which can then be used for critical downstream tasks such as control.

\section{Code and Data Availability}
The code used for the numerical experiments is available at: https://github.com/fisherng19/Stable-Linear-Takens-Embedding-Theorem-with-Unevenly-Spaced-Delays.git

\section{Acknowledgment}
We acknowledge support from the Air Force Office of Scientific Research  (FA9550-24-1-0141).

\bibliographystyle{IEEEtran}
\bibliography{references}

\begin{IEEEbiography}[{\includegraphics[width=1in,height=1.25in,clip,keepaspectratio]{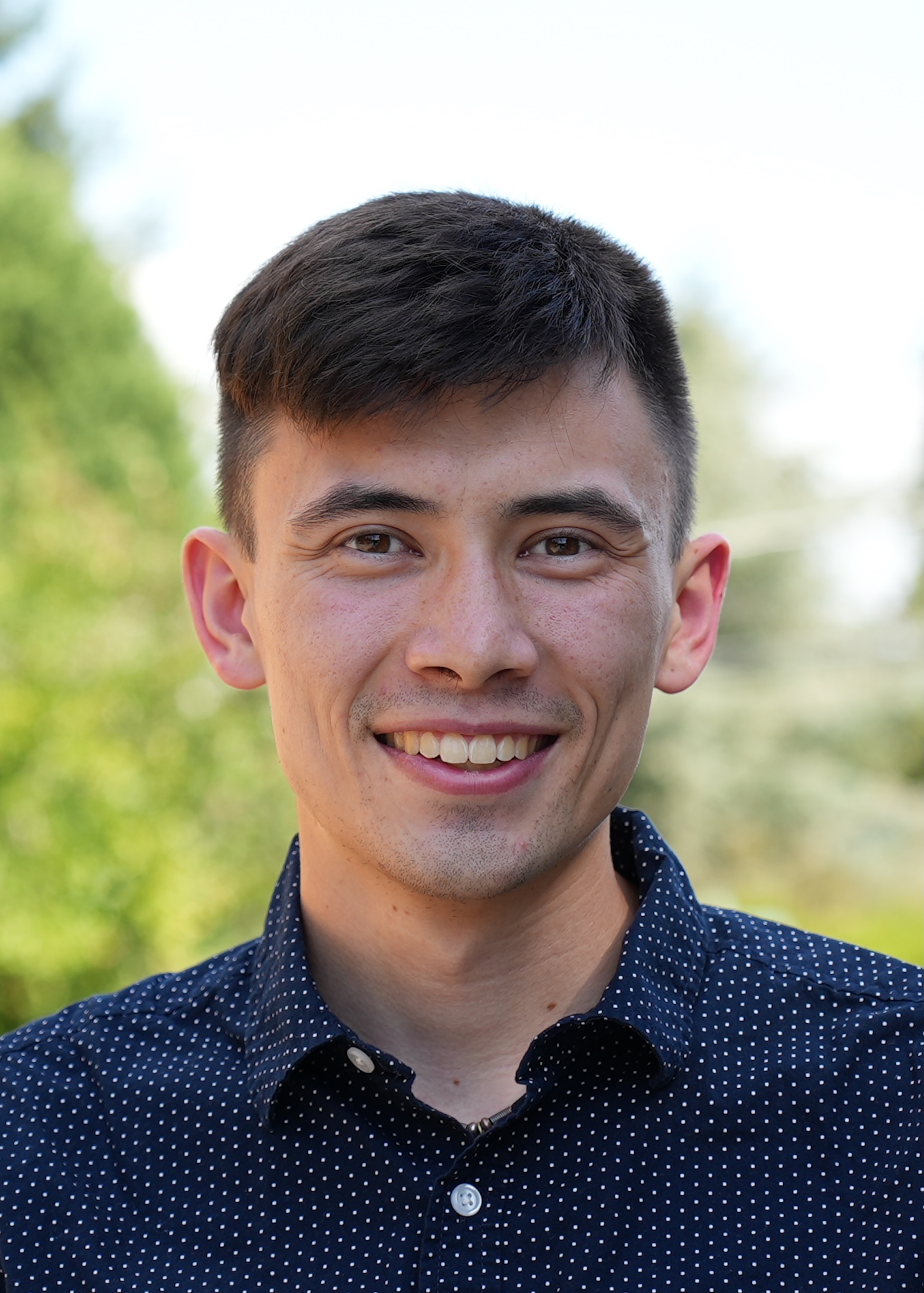}}]{Fisher Ng} is a Ph.D. candidate in the Department of Applied Mathematics at the University of Washington, Seattle, WA. He received the B.S. degree in applied mathematics and the B.S. degree in mechanical engineering from Gonzaga University, Spokane, WA, USA, in 2022, the M.S. degree in applied mathematics from the University of Washington, Seattle, WA, USA, in 2025, and is currently pursuing his Ph.D. degree at the University of Washington, Seattle, WA, USA. His research interests include time-delay embedding theory, dynamical systems, and multi-scale and reduced-order modeling.
\end{IEEEbiography}

\begin{IEEEbiography}[{\includegraphics[width=1in,height=1.25in,clip,keepaspectratio]{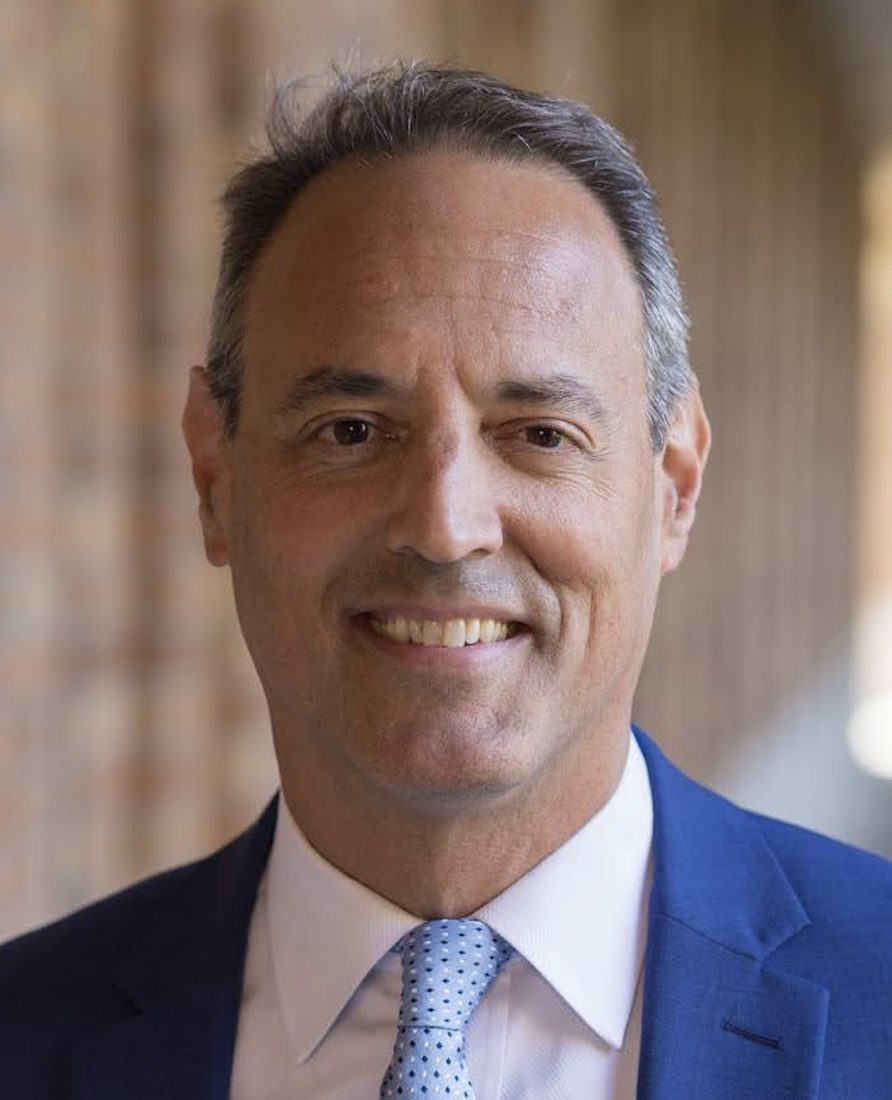}}]{J. Nathan Kutz}
is Director of Physics Informed AI at Autodesk Research in London, UK.  He is on-leave from the University of Washington where he was the Boeing Professor of Applied Mathematics and Electrical and Computer Engineering, having served as director of the AI Institute in Dynamic Systems from 2020-2025 and chair of applied mathematics from 2007-2015.  He received the B.S. degree in physics and mathematics from the University of Washington in 1990 and the Ph.D. in applied mathematics from Northwestern University in 1994. He was a postdoc in the applied and computational mathematics program at Princeton University before taking his faculty position. He joined Autodesk Research in January 2026.  He has a wide range of interests, including neuroscience to fluid dynamics where he integrates machine learning with dynamical systems and control.\end{IEEEbiography}

\end{document}